\documentclass[sn-mathphys-num]{sn-jnl}

\usepackage{graphicx}%
\usepackage{multirow}%
\usepackage{amsmath,amssymb,amsfonts}%
\usepackage{amsthm}%
\usepackage{mathrsfs}%
\usepackage[title]{appendix}%
\usepackage{xcolor}%
\usepackage{textcomp}%
\usepackage{manyfoot}%
\usepackage{booktabs}%
\usepackage{algorithm}%
\usepackage{algorithmicx}%
\usepackage{algpseudocode}%
\usepackage{listings}%

\usepackage{float}
\usepackage{subfig}
\usepackage{caption}
\usepackage{makecell}

\theoremstyle{thmstyleone}%
\newtheorem{theorem}{Theorem}
\newtheorem{corollary}{Corollary}
\newtheorem{lemma}{Lemma}
\theoremstyle{thmstyletwo}%
\newtheorem{remark}{Remark}%

\theoremstyle{thmstylethree}%
\newtheorem{assumption}{Assumption}%

\newcommand{\order}[1]{\mathcal{O}\left(#1\right)}
\newcommand{\torder}[1]{\tilde{\mathcal{O}}\left(#1\right)}
\newcommand{\torderi}[1]{\tilde{\mathcal{O}}(#1)}
\newcommand{\prt}[1]{\left(#1\right)}
\newcommand{\brk}[1]{\left[#1\right]}
\newcommand{\crk}[1]{\left\{#1\right\}}

\newcommand{\orderi}[1]{\mathcal{O}(#1)}
\newcommand{\prti}[1]{(#1)}
\newcommand{\brki}[1]{[#1]}
\newcommand{\crki}[1]{\{#1\}}

\newcommand{\inpro}[1]{\left\langle #1 \right\rangle}
\newcommand{\inproi}[1]{\langle #1 \rangle}
\newcommand{\condE}[2]{\E\brk{#1\middle|#2}}
\newcommand{\condEi}[2]{\E[#1|#2]}
\newcommand{\norm}[1]{\left\Vert #1 \right\Vert}
\newcommand{\normi}[1]{\Vert #1 \Vert}

\newcommand{\R}{\mathbb{R}}
\newcommand{\E}{\mathbb{E}}

\newcommand{\x}{\mathbf{x}}
\newcommand{\z}{\mathbf{z}}
\newcommand{\g}{\mathbf{g}}
\newcommand{\1}{\mathbf{1}}

\newcommand{\T}{\intercal}
\newcommand{\sumn}{\sum_{i=1}^n}
\newcommand{\cN}{\mathcal{N}}

\newcommand{\y}{\mathbf{y}}

\newcommand{\cF}{\mathcal{F}}

\newcommand{\cL}{\mathcal{L}}
\newcommand{\0}{\mathbf{0}}
\newcommand{\Barx}[1]{\1\bar{x}^{\T}_{#1}}

\newcommand{\tpi}{\tilde{\Pi}}

\newcommand{\tpx}{\tpi\tx}
\newcommand{\tpy}{\tpi\ty}
\newcommand{\tW}{\tilde{W}}
\newcommand{\zsh}[1]{(\z_{#1})_{\#}}
\newcommand{\ysh}[1]{(\y_{#1})_{\#}}

\newcommand{\uf}{f^*}
\newcommand{\trw}{\tilde{\rho}_w}
\newcommand{\cR}{\mathcal{R}}
\newcommand{\xsh}[1]{(\x_{#1})_{\#}}
\newcommand{\tx}{\tilde{\mathbf{x}}}
\newcommand{\ty}{\tilde{\mathbf{y}}}

\newcommand{\ub}[1]{\underline{#1}}
\newcommand{\wb}[1]{\bar{\underline{#1}}}
\newcommand{\sigfn}{\sigma_f^*}
\newcommand{\cC}{\mathcal{C}}
\newcommand{\cH}{\mathcal{H}}

\newcommand{\cA}{\mathcal{A}}
\newcommand{\cS}{\mathcal{S}}
\newcommand{\KTN}{K_{\text{Transient}}^{(\text{NCVX})}}
\newcommand{\KTP}{K_{\text{Transient}}^{(\text{PL})}}

\definecolor{cuhkpl}{RGB}{152,24,147}

\begin{document}

\title[Distributed Stochastic Momentum Tracking]{An Accelerated Distributed Stochastic Gradient Method with Momentum}


\author[1]{\fnm{Kun} \sur{Huang}}\email{kunhuang@link.cuhk.edu.cn}

\author*[1]{\fnm{Shi} \sur{Pu}}\email{pushi@cuhk.edu.cn}

\author[2]{\fnm{Angelia} \sur{Nedi\'c}}\email{angelia.nedich@asu.edu}

\affil*[1]{\orgdiv{School of Data Science (SDS)}, \orgname{The Chinese University of Hong Kong, Shenzhen}, \orgaddress{\street{2001 Longxiang Boulevard}, \city{Shenzhen}, \postcode{518172}, \state{Guangdong}, \country{China}}}

\affil[2]{\orgdiv{School of Electrical, Computer and Energy Engineering}, \orgname{Arizona State University}, \orgaddress{\street{1151 S Forest Ave}, \city{Tempe}, \postcode{85281}, \state{Arizona}, \country{United States}}}



\abstract{In this paper, we introduce an accelerated distributed stochastic gradient method with momentum
for solving the distributed optimization problem, where a group of $n$ agents collaboratively minimize the average of the local objective functions over a connected network. 
The method, termed ``Distributed Stochastic Momentum Tracking (DSMT)'', is a single-loop algorithm that utilizes the momentum tracking technique as well as the Loopless Chebyshev Acceleration (LCA) method.
We show that DSMT can asymptotically achieve comparable convergence rates as centralized stochastic gradient descent (SGD) method under a general variance condition regarding the stochastic gradients. Moreover, the number of iterations (transient times) required for DSMT to achieve such rates behaves as $\orderi{n^{5/3}/(1-\lambda)}$ for minimizing general smooth objective functions, and $\orderi{\sqrt{n/(1-\lambda)}}$ under the Polyak-{\L}ojasiewicz (PL) condition. Here, the term $1-\lambda$ denotes the spectral gap of the mixing matrix related to the underlying network topology.
Notably, the obtained results do not rely on multiple inter-node communications or stochastic gradient accumulation per iteration, and the transient times are the shortest under the setting to the best of our knowledge.
}




\maketitle

\section{Introduction}
We investigate how a group of networked agents $\cN:=\crki{1,2,\ldots, n}$ collaborate to solve the following distributed optimization problem:
\begin{equation}
    \label{eq:P}
    \min_{x\in\R^p} f(x):= \frac{1}{n}\sum_{i=1}^n f_i(x),
\end{equation}
where each agent $i$ has access only to its local objective function $f_i:\R^p\rightarrow\R$. Problem \eqref{eq:P} arises in various fields, including signal processing, distributed estimation, and machine learning. Specifically, solving large-scale problems such as training foundation models \cite{brown2020language,chowdhery2022palm,zhang2022opt,openai2023gpt4} in specialized and centralized clusters can be costly \cite{yuan2022decentralized,wang2023cocktailsgd}, while decentralized training offers a promising alternative that relieves the robustness bottleneck and reduces the high latency at the central server \cite{nedic2018network}. 
In this paper, we consider the typical setting where each agent can query unbiased noisy gradients of $f_i(x)$ in the form of $g_i(x;\xi)$ that satisfies a certain variance condition, where $\xi$ represents a random sample or a batch of random samples.

Decentralization may slow down the optimization process due to the partial communication over sparse networks. For instance, the number of iterations required for the distributed subgradient descent (DGD) method \cite{nedic2009distributed} to reach certain accuracy threshold significantly grows with the network size \cite{pu2020asymptotic}, downgrading its performance compared to the centralized subgradient method. When stochastic gradients are used, however, several distributed methods such as those in \cite{lian2017can,pu2021sharp,huang2021improving,tang2018d,spiridonoff2020robust,pu2021distributed,alghunaim2021unified,yuan2021removing,xin2021improved,huang2023cedas,koloskova2019decentralized} have exhibited the so-called ``asymptotic network independent (ANI)'' property \cite{pu2020asymptotic}. In other words, these methods achieve the same convergence rates as their centralized counterparts after a finite transient time has passed. Such a property ensures that the required number of iterations to achieve high accuracy do not increase significantly with the network size. Nevertheless, it is still critical to develop algorithms with shorter transient times that typically depend on the network size and topology. 
Currently, the best known transient times for minimizing general smooth objective functions are $\orderi{n^3/(1-\lambda)^2}$ \cite{alghunaim2021unified,huang2023cedas} and $\orderi{n/(1-\lambda)}$ \cite{huang2021improving,alghunaim2021unified,yuan2021removing} with and without the strong convexity (or the Polyak-{\L}ojasiewicz condition) on the objective functions, respectively.

Chebyshev Acceleration (CA) has been a popular technique for achieving faster consensus among networked agents during the distributed optimization process \cite{scaman2017optimal}. When CA is combined with stochastic gradient accumulation using a large batch size at every iteration, the optimal convergence guarantee can be achieved \cite{lu2021optimal,yuan2022revisiting}.
However, CA requires inner loops of multiple communication steps which may affect the practical performance of the algorithms, and large batches are not always available.
Inspired by the recently developed Loopless Chebyshev Acceleration (LCA) technique that works without inner loops \cite{song2021optimal}, we consider in this paper a new method termed ``Distributed Stochastic Momentum Tracking (DSMT)'' that incorporates two key features. First, the LCA technique is utilized to accelerate the communication process and avoids multiple communication steps between two successive stochastic gradient computations. Second, each agent employs an auxiliary variable to track the average momentum parameter over the network.
 We show that DSMT improves the transient times over the existing works for minimizing smooth objective functions, with or without the Polyak-{\L}ojasiewicz (PL) condition. In particular, the convergence rates of DSMT are derived under the most general condition on the stochastic gradients under the distributed setting \cite{khaled2020better,li2022unified,lei2019stochastic}, enhancing the applicability of the method.
 
It is worth noting that, the convergence result of DSMT implies the momentum parameter can benefit the algorithmic convergence of distributed stochastic gradient methods, particularly related to the network topology. Such an observation can potentially inspire future development of distributed optimization algorithms.

\subsection{Related Works}

There exists a rich literature on the distributed implementation of stochastic gradient methods over networks. For example, early works including \cite{chen2012limiting,morral2017success,chen2015learning,pu2017flocking,pu2018swarming} suggest that distributed stochastic gradient methods may achieve comparable performance with centralized SGD under specific settings. More recently, the paper \cite{lian2017can} first demonstrates the ANI property enjoyed by the distributed stochastic gradient descent (DSGD) method, followed by the study in \cite{pu2021sharp,koloskova2019decentralized,pu2020asymptotic}. Subsequent research that considers gradient tracking based methods \cite{xin2021improved,pu2021distributed,koloskova2021improved,xiao2023one,ye2022snap} and primal-dual like methods \cite{huang2021improving,yuan2021removing,tang2018d,huang2023cedas} also demonstrate the ANI property while relieving from the data heterogeneity challenge encountered by DSGD. Both types of methods are unified in \cite{alghunaim2021unified}. 

Recently, there exists a line of works that achieve the optimal iteration complexity in distributed stochastic optimization by integrating stochastic gradient accumulation and accelerated communication techniques such as Chebyshev Acceleration (CA) \cite{yuan2022revisiting,lu2021optimal}. 
In particular, the works in \cite{yuan2022revisiting} and \cite{lu2021optimal} obtain the optimal iteration complexity under smooth objective functions with and without the PL condition, respectively. 
For finite sum problems, the work in \cite{luo2022optimal} achieves the optimal guarantee by further incorporating a variance reduction technique.\footnote{It is worth mentioning that when the objective functions take a finite sum form, distributed random reshuffling (RR) methods can further improve the iteration complexity \cite{huang2023drr,huang2023abcrr}.}
Nevertheless, CA that relies on inner loops of multiple communication steps may not be communication-efficient, and sampling large batches for computing the stochastic gradients is not always practical \cite{xiao2023one}.
Therefore, we are motivated to develop an efficient distributed stochastic gradient method that samples constant batches for computing the stochastic gradients and performs one round of communication at every iteration. 

The momentum method \cite{polyak1964some,ghadimi2020single} is popular for accelerating the convergence of first-order methods. Recently, it has been applied to the distributed setting \cite{gao2023distributed,yuan2021decentlam,wang2021distributed,lin2021quasi,xiao2023one}, resulting in practical improvements. However, it remains unclear whether momentum parameters can benefit the convergence of distributed stochastic gradient methods theoretically.

\subsection{Main Contribution}

The main contribution of this paper is two-fold. 

Firstly, the proposed DSMT method achieves improved convergence rates and better transient times over the existing distributed stochastic gradient methods (see Table \ref{tab:ncvx}). Specifically, for general smooth objective functions, DSMT shortens the state-of-the-art transient time from $\mathcal{O}(n^3/(1-\lambda)^2)$ to $\mathcal{O}(n^{5/3}/(1-\lambda))$, where $1-\lambda$ denotes the spectral gap of the mixing matrix that may scale as $\mathcal{O}(1/n^2)$ for sparse networks. When the objective functions further satisfy the PL condition, DSMT shortens the state-of-the-art transient time from $\mathcal{O}(n/(1-\lambda))$ to $\mathcal{O}(\sqrt{n/(1-\lambda)})$.

Secondly, the convergence results of DSMT only assumes the so-called ABC condition \cite{khaled2020better,li2022unified,lei2019stochastic} on the stochastic gradients, which is the most general variance condition under the distributed settings to our knowledge. By comparison, existing works depend on the relaxed growth condition (RGC) or the bounded variance \cite{bottou2018optimization} regarding the stochastic gradients or the bounded gradient dissimilarity (BGD) assumption concerning the data heterogeneity among the agents. 
Such a generalization expands the practical applicability of DSMT given that the ABC condition covers counterexamples that do not satisfy RGC \cite{huang2023distributed,khaled2020better}.

\begin{table}[]
    \setlength{\tabcolsep}{0.5pt}
    \centering
    \begin{tabular}{@{}ccccc@{}}
        \toprule
        Method                                                        & g                     & \makecell[c]{Additional\\ Assumption} & \makecell[c]{Transient Time\\ Nonconvex}                                    & \makecell[c]{Transient Time\\ PL condition}                                \\ \midrule
        \makecell[c]{DSGD \cite{lian2017can}}                         & $(\sigma^2, 0)$       & $(\zeta^2, 0)$-BGD                    & $\order{\frac{n^3}{(1-\lambda)^4}}$                                         & /                                          \\
        \makecell[c]{DSGD \cite{pu2021sharp}}                         & $(\sigma^2, \eta^2)$  & /                    & /                                                                           & $\order{\frac{n}{(1-\lambda)^2}}$                                          \\
        \makecell[c]{ED \cite{alghunaim2021unified}}                  & $(\sigma^2, 0)$       & /                                     & $\order{\frac{n^3}{(1-\lambda)^2}}$                                         & $\order{\frac{n}{1-\lambda}}$                                              \\
        \makecell[c]{EDAS\cite{huang2021improving}}                   & $(\sigma^2, \eta^2)$  & /                                     & /                                                                           & $\order{\frac{n}{1-\lambda}}$                                              \\
        \makecell[c]{DSGT \cite{alghunaim2021unified}}                & $(\sigma^2, 0)$       & /                                     & $\order{\max\crk{\frac{n^3}{(1-\lambda)^2}, \frac{n}{(1-\lambda)^{8/3}} }}$ & $\order{\max\crk{\frac{n}{1-\lambda}, \frac{n^{1/3}}{(1-\lambda)^{4/3}}}}$ \\
        { DSGT \cite{koloskova2021improved}}                       & {$(\sigma^2, 0)$} & { /}                               & {$\order{\frac{n^3}{c^4(1-\lambda)^2}}$}                                & {$\order{\frac{n}{c^2(1-\lambda)}}$}                                   \\
        DSGT-HB \cite{gao2023distributed}                             & $(\sigma^2, 0)$       & /                                     & $\order{\frac{n^3}{(1-\lambda)^4}}$                                         & /                                                                          \\
        QG-DSGDm \cite{lin2021quasi}                                  & $(\sigma^2, 0)$       & $(\zeta^2, 0)$-BGD                    & $\order{\frac{n^3}{(1-\lambda)^4}}$                                         & /                                                                          \\
        { Prox-DASA-GT \cite{xiao2023one}}                         & {$(\sigma^2, 0)$} & { /}                               & {$\order{\frac{n^3}{(1-\lambda)^4}}$}                                   & { /}                                                                    \\
        \makecell[c]{Momentum\\ Tracking \cite{takezawa2022momentum}} & $(\sigma^2, 0)$       & /                                     & $\order{\frac{n^3}{(1-\lambda)^8}}$                                         & /                                                                          \\ \midrule
        \makecell[c]{DSMT\\ \textbf{(This paper)}}                    & \textbf{ABC}          & $f_i$ lower bounded                   & $\boldsymbol{\order{\frac{n^{5/3}}{1-\lambda}}}$                            & $\boldsymbol{\order{\sqrt{\frac{n}{1-\lambda}}}}$                          \\ \bottomrule
        \end{tabular}
    \caption{Comparison of different methods regarding the assumptions and the transient times for minimizing smooth objective functions. The ``g'' column describes the conditions assumed on the stochastic gradients, where the notation $(\sigma^2, \eta^2)$ represents the relaxed growth condition, i.e., $\condEi{\normi{g_i(x;\xi_{i}) - \nabla f_i(x)}^2}{x}\leq \sigma^2 + \eta^2\norm{\nabla f_i(x)}^2$, $\forall i\in\cN$, and ABC stands for Assumption \ref{as:abc}.
    The notation $(\zeta^2, \psi^2)$-BGD stands for the bounded gradient dissimilarity assumption requiring that $\frac{1}{n}\sumn\norm{\nabla f_i(x) - f(x)}^2\leq \zeta^2 + \psi^2\normi{\nabla f(x)}^2$. { The constant $c$ in \cite{koloskova2021improved} depends on the mixing matrix $W$.}
    }
        \label{tab:ncvx}
\end{table}

\subsection{Notation and Assumptions}

Throughout this paper, column vectors are considered by default unless specified otherwise. We use $x_{i,k}\in\R^p$ to denote the iterate of agent $i$ at the $k$-th iteration. For the sake of clarity and presentation, we introduce the stacked variables as follows:
\begin{align*}
    \x_k&:= (x_{1,k}, x_{2, k},\ldots, x_{n,k})^{\T}\in\R^{n\times p},\\
    \nabla F (\x_k) &:= \prt{\nabla f_1(x_{1,k}),\nabla f_2(x_{2,k}),\ldots, \nabla f_n(x_{n,k})}^{\T}\in\R^{n\times p},\\
    A_{\#}&:=\begin{pmatrix}
        A\\
        A
    \end{pmatrix} \in\R^{2n\times p}.
\end{align*}
We use $\bar{x}\in\R^p$ to denote the averaged variables of $x_i$ among the agents. For instance, the variable $\bar{x}_k:= 1/n\sumn x_{i, k}$ stands for the average of all the agents' iterates at the $k$-th iteration. 
We use $\normi{\cdot}$ to denote the Frobenius norm for a matrix by default and the $\ell_2$ norm for a vector. The term $\inproi{a, b}$ stands for the inner product of two vectors $a, b\in\R^{p}$. For two matrices $A, B\in\R^{n\times p}$, $\inpro{A, B}$ is defined as $\inproi{A, B} := \sum_{i=1}^n\inproi{A_i, B_i}$, where $A_i$ (and $B_i$) represents the $i$-row of $A$ (and $B$).

We next introduce the standing assumptions. Assumption \ref{as:smooth} is common that requires the objective functions to be smooth and lower bounded.
\begin{assumption}
    \label{as:smooth}
    Each $f_i(x):\R^p\rightarrow\R$ is $L$-smooth, i.e., 
    \begin{align*}
        \norm{\nabla f_i(x) - \nabla f_i(x')}\leq L\norm{x - x'}, \ \forall x,x'\in\R^p,
    \end{align*}
    and bounded from below, i.e., $f_i(x)\geq \uf_i := \inf_x f_i(x)>-\infty, \forall x\in\R^p$. 
\end{assumption}

We let $f^*:=\inf_x f(x)$. Also,
let the filtration $\crk{\cF_k}_{k\geq 0}$ be generated by $\crki{\xi_{i,\ell}| i\in\cN, \ell = 0,1,\ldots, k-1}$. Regarding the stochastic gradients, we consider the ABC condition as given in the following assumption.
\begin{assumption}
    \label{as:abc}
    Assume each agent has access to an conditionally unbiased stochastic gradient $g_i(x;\xi_i)$ of $\nabla f_i(x)$, i.e., $\condEi{g_i(x;\xi_{i,k})}{\cF_k} = \nabla f_i(x)$, and there exist constants $C,\sigma\geq 0$ such that for any $k\in\mathbb{N}$, $i\in\cN$,
    \begin{align}
        \label{eq:abc}
        \condE{\norm{g_{i}(x;\xi_{i,k}) - \nabla f_i(x)}^2}{\cF_k}&\leq C\brk{f_i(x) - \uf_i} + \sigma^2.
    \end{align} 
    Additionally, the stochastic gradients are independent across different agents for all $k\geq 0$. 
\end{assumption}
It is noteworthy that the relaxed growth condition (RGC) implies \eqref{eq:abc} given that each $f_i$ has Lipschitz continuous gradients and is lower bounded \cite{khaled2020better}, however, there exist examples that only satisfy the ABC condition but not RGC \cite{khaled2020better,huang2023distributed}. Moreover, Assumption \ref{as:abc} generally holds for the empirical risk minimization problems if the stochastic gradients are queried through uniformly sampling with replacement \cite{huang2023distributed}.

We assume the agents in the network are connected via a graph $\mathcal{G} = (\cN, \mathcal{E})$ with $\mathcal{E}\subseteq \cN\times \cN$ representing the set of edges connecting the agents. In particular, $(i,i)\in\mathcal{E}$ for all $i\in\mathcal{N}$. The set of neighbors for agent $i$ is denoted by $\mathcal{N}_i=\{j\in \mathcal{N}:(i,j)\in \mathcal{E}\}$.
The element $w_{ij}$ in the weight matrix $W\in\mathbb{R}^{n\times n}$ represents the weight of the edge between agents $i$ and $j$. 
Regarding the network topology, we consider Assumption \ref{as:graph} that is standard in the distributed optimization literature. The condition guarantees that the spectral norm $\lambda$ of the matrix $(W - \1\1^{\T}/n)$ is strictly less than one. 

\begin{assumption}
    \label{as:graph}
    The graph $\mathcal{G}$ is undirected and {connected}, i.e., there exists a path between any two nodes in $\mathcal{G}$. There is a direct link between $i$ and $j$ $(i\neq j)$ in $\mathcal{G}$ if and only if $w_{ij}>0$ and $w_{ji}>0$; otherwise, $w_{ij}=w_{ji}=0$. The mixing matrix $W$ is nonnegative and doubly stochastic, i.e., $\1^{\T}W=\1^{\T}$ and $W\1=\1$. In addition, we assume $w_{ii}>0$ for all $i\in\cN$.
\end{assumption} 

\subsection{Organization}

The rest of this paper is organized as follows. In Section \ref{sec:dsmt}, we introduce the proposed DSMT algorithm starting with its motivation. We then proceed to conduct a preliminary analysis in Section \ref{sec:pre}. The main convergence results of DSMT under smooth objective functions with or without the PL condition are presented in Section \ref{sec:main}. We then provide numerical experiments in Section \ref{sec:sim} and conclude the paper in Section \ref{sec:con}.

\section{A Distributed Stochastic Momentum Tracking Method}
\label{sec:dsmt}

In this section, we introduce the proposed algorithm, Distributed Stochastic Momentum Tracking (DSMT). We start with the motivation and the intuition behind DSMT.

\subsection{Motivation}
\label{subsec:motivation}

The proposed algorithm relies heavily on the so-called Loopless Chebyshev Acceleration (LCA) technique stated in Lemma \ref{lem:lca} below. 
\begin{lemma}[Lemma 11 in \cite{song2021optimal}]
    \label{lem:lca}
    Suppose the mixing matrix $W$ is symmetric and positive semidefinite. Define $\eta_w := 1/(1 + \sqrt{1 - \lambda^2})$, then $\trw:= \sqrt{\eta_w}\sim\orderi{1-\sqrt{1-\lambda^2}}$, and for any $A\in\R^{n\times p}$ and $k\geq 0$,
    \begin{align*}
        \norm{\tpi\tW^k\tpi A_{\#}}^2 \leq c_0 \trw^{2k}\norm{\Pi A}^2, \ c_0 = 14,
    \end{align*}
	where 
	\begin{align*}
		\tpi:= \begin{pmatrix}
			\Pi & \0 \\
			\0 & \Pi
		\end{pmatrix},\ \Pi:= I - \frac{\1\1^{\T}}{n},\ A_{\#}:= \begin{pmatrix}
			A\\
			A
		\end{pmatrix},\ \tW := \begin{pmatrix}
			(1 + \eta_w)W & -\eta_w I\\
			I & \0
		\end{pmatrix}.
	\end{align*}
\end{lemma}
The above lemma implies that, by employing an augmented mixing matrix $\tW$, the consensus rate among multiple agents can be accelerated compared to applying the basic consensus procedure using the matrix $W$.
In addition, when LCA is employed in distributed optimization algorithms, there is no need to perform inner loops of multiple communication rounds. For instance, the Optimal Gradient Tracking (OGT) method proposed in \cite{song2021optimal} achieves the optimal communication and computation complexities simultaneously without inner loops of multiple communication rounds for minimizing smooth strongly convex objective functions using full gradients.

In the following discussion, we adhere to the notations introduced in Lemma \ref{lem:lca}. Additionally, we use $[\tx_k]_{1:n}$ to represent the first $n$ rows of $\tx_k\in\R^{2n\times p}$.

It is noteworthy that applying the LCA technique to distributed stochastic gradient methods while achieving improved transient times is a nontrivial task. 
A recent attempt in \cite{ye2022snap} realized the same transient time $\orderi{n/(1-\lambda)}$ as those presented in \cite{huang2021improving,alghunaim2021unified,yuan2021removing}. We discuss the reasons of this result as follows.
In \cite{ye2022snap}, the LCA technique is directly combined with the Distributed Stochastic Gradient Tracking (DSGT) method described in \eqref{eq:DSGT}, where
$\g_k:= \prti{g_1(x_{1,k};\xi_{1,k}),\ldots, g_n(x_{n,k};\xi_{n,k})}^{\T}\in\R^{n\times p}$ denotes the stacked stochastic gradient at the $k$-th iteration.
\begin{subequations}
    \label{eq:DSGT}
    \begin{align}
        \x_{k + 1} &= W\prt{\x_k - \alpha \y_k},\\
        \y_{k + 1} &= W\prt{\y_k + \g_{k + 1} - \g_k},\ \y_0 = \g_0\label{eq:dsgt_y}.
    \end{align}
\end{subequations}
Here, the auxiliary variables $\y_k$ track the averaged stochastic gradients among the agents since $\bar{y}_k = \bar{g}_k$ based on Assumption \ref{as:graph} and induction from \eqref{eq:dsgt_y}.
One common approach for analyzing the convergence of DSGT involves estimating the following term (see, e.g., \cite{pu2021distributed,ye2022snap} or \eqref{eq:lya_hk}):
\begin{equation}
    \label{eq:cons_dsgt}
    \E\brk{\norm{\x_k - \1\bar{x}_k^{\T}}^2} + \frac{\alpha^2 \cC_0}{(1-\lambda)^2}\E\brk{\norm{\y_k - \1\bar{y}_k^{\T}}^2},
\end{equation}
for some constant $\cC_0>0$ that is independent of $n$ and $(1-\lambda)$. However, this leads to an unsatisfactory performance of the algorithm, primarily due to the bias in $\y_k$, i.e., $\E[\y_k - \g_k|\cF_k]\neq 0$.


To alleviate the bias in $\y_k$, we introduce alternative variables, $\z_k$, for $\y_k$ to track. Specifically, the update $\y_k$ is as follows:
\begin{equation}
    \label{eq:dsmt_yk0}
    \y_{k + 1} = W\prt{\y_k + \z_{k + 1} - \z_k}.
\end{equation} 

Simultaneously, the variables $\z_k$ follow the update defined in \eqref{eq:dsmt_zk0}:
\begin{equation}
    \label{eq:dsmt_zk0}
    \z_{k + 1} - \z_k =-(1 - \beta)\prt{\z_k - \nabla F(\x_{k + 1})} + (1 - \beta) \brk{\g_{k + 1} - \nabla F(\x_{k + 1})},
\end{equation}
where $\beta\in(0,1)$.
With the help of the additional term $(1-\beta)$ in \eqref{eq:dsmt_zk0}, the impact of the stochastic gradient variance is better controlled, as detailed in Lemma \ref{lem:cons_R}. Such a strategic adjustment enables the cancellation of the extra factor $1/(1-\lambda)$  in~\eqref{eq:cons_dsgt}.

\begin{remark}
    \label{rem:sgdm}
    Substituting the updates \eqref{eq:dsmt_yk0} and \eqref{eq:dsmt_zk0} into the DSGT method \eqref{eq:DSGT} leads to the following relations. 
    \begin{equation}
        \label{eq:smt}
        \begin{aligned}
            \x_{k + 1} &= W\prt{\x_k - \alpha \y_k},\\
            \z_{k + 1} &= \beta \z_k + (1-\beta)\g_{k + 1},\\
            \y_{k + 1} &= W\prt{\y_k + \z_{k + 1} - \z_k},\ \y_0 = \z_0 = (1-\beta)\g_0.
        \end{aligned}
    \end{equation}
    Invoking Assumption \ref{as:graph}, we can deduce that $\bar{y}_k = \bar{z}_k$ by induction. 
    Consequently, the behavior of $\bar{x}_k$ as described in \eqref{eq:moti_m} resembles the update of the Stochastic Gradient Descent with Momentum (SGDM) method \cite{polyak1964some}. It turns out that $\y_k$ tracks the momentum variable $\z_k$, and this is why the proposed method is named ``distributed stochastic momentum tracking''.
    
    \begin{equation}
        \label{eq:moti_m}
        \begin{aligned}
            \bar{x}_{k + 1} &= \bar{x}_k - \alpha \bar{z}_k,\\
            \bar{z}_{k + 1} &= \beta \bar{z}_k + (1-\beta)\bar{g}_k.
        \end{aligned}
    \end{equation}
    { Notably, setting $\beta = 1-\alpha$ recovers the averaging technique used in stochastic optimization literature, such as \cite{ghadimi2020single,xiao2023one}.}
\end{remark}

The preceding discussion motivates the proposed method given in the following section.

\subsection{The Proposed Method}

We begin by presenting the compact form of the Distributed Stochastic Momentum Tracking (DSMT) method in \eqref{eq:DSMT}.
In essence, DSMT leverages the Loopless Chebyshev Acceleration (LCA) technique for the update \eqref{eq:smt}.
\begin{subequations}
    \label{eq:DSMT}
    \begin{align}
        \tx_{k + 1} &= \tW\prt{\tx_k - \alpha \ysh{k}}\label{eq:tx}\\
        \x_{k + 1} &= [\tx_{k + 1}]_{1:n}\label{eq:xn} \\
        \z_{k + 1} &= \beta\z_k + (1-\beta)\g_{k+1}\label{eq:z}\\
        \ty_{k + 1} &= \tW\prt{\ty_k + \zsh{k + 1} - \zsh{k}}\label{eq:ty}\\
        \y_{k + 1} &= [\ty_{k + 1}]_{1:n}\label{eq:yn},
    \end{align}
\end{subequations}
where $\tW$ defined in Lemma \ref{lem:lca} signifies the application of the LCA technique,
and $\ysh{k} := (\y_k^{\T}, \y_k^{\T})^{\T}\in\R^{2n\times p}$. 
Consequently, the updates in \eqref{eq:xn} and \eqref{eq:yn} are required to obtain $\x_{k + 1}$ and $\ysh{k+1}$ in the subsequent iterations. The updates in \eqref{eq:z} and \eqref{eq:ty} align with the rationale explained for \eqref{eq:smt}, representing a critical step in improving the convergence rate compared to the method outlined in \cite{ye2022snap}.

\begin{algorithm}
	\begin{algorithmic}[1]
		\State Initialize $x_{i,0} = x_{i,0}^l\in\R^p$ for all agent $i\in\mathcal{N}$, determine $W = [w_{ij}]\in\R^{n\times n}$, stepsize $\alpha$ and parameter $\beta$. Initialize $y_{i,0} = y_{i,0}^l = z_{i,0} = (1-\beta)g_{i,0}$. Input $\eta_w$.
		\For{$k=0, 1, 2, ..., K-1$}
		\For{Agent $i = 1, 2, ..., n$ in parallel}
		\State $x_{i, k + \frac{1}{2}} = x_{i,k} - \alpha y_{i,k}$, $x_{i, k + \frac{1}{2}}^l = x_{i, k}^l - \alpha y_{i,k}$. \label{line:dsmt_x}
		\State $x_{i, k + 1} = (1 + \eta_w)\sum_{j\in\cN_i}w_{ij} x_{j, k + \frac{1}{2}} - \eta_w x_{i,k + \frac{1}{2}}^l$, $x_{i, k + 1}^l = x_{i, k + \frac{1}{2}}$. \label{line:dsmt_xcom}
		\State $g_{{i}, k + 1} = g_i(x_{i, k + 1};\xi_{i, k + 1})\in\R^p$.\label{line:dsmt_g}
		\State $z_{i, k + 1} = \beta z_{i, k} + (1-\beta) g_{i, k + 1}$. \label{line:dsmt_z}
		\State $y_{i, k + \frac{1}{2}} = y_{i, k} + z_{i, k + 1} - z_{i,k}$, $y_{i, k + \frac{1}{2}}^l = y_{i, k}^l + z_{i, k + 1} - z_{i,k}$. \label{line:dsmt_y}
		\State $y_{i, k + 1} = (1 + \eta_w) \sum_{j\in\cN_i} w_{ij}y_{j, k + \frac{1}{2}} - \eta_w y_{i, k+\frac{1}{2}}^l$, $y_{i, k + 1}^l = y_{i, k + \frac{1}{2}}$.\label{line:dsmt_ycom}
		\EndFor
		\EndFor
        \State Output $x_{i,T}$ for all agent $i\in\cN$.
	\end{algorithmic}
	\caption{Distributed Stochastic Momentum Tracking (DSMT)}
	\label{alg:dsmt}
\end{algorithm}

The formal description of the DSMT method is outlined in Algorithm \ref{alg:dsmt}. 
At each iteration, agent $i$ first performs an approximate gradient descent step in Line \ref{line:dsmt_x} to obtain $x_{i,k+\frac{1}{2}}$ and $x_{i,k+\frac{1}{2}}^l$. Then, after communicating with neighboring agents the updated $x_{i,k+\frac{1}{2}}$ in Line \ref{line:dsmt_xcom}, the new iterates $x_{i,k+1}$ and $x_{i,k+1}^l$ are obtained. Subsequently, agent $i$ queries a noisy gradient based on $x_{i,k + 1}$ and completes the update of $z_{i, k + 1}$ in Line \ref{line:dsmt_z}. The remaining procedures aim to track the averaged variable $\bar{z}_k$ given the initialization $\y_0 = \z_0 = (1-\beta)\g_0${, which necessitates a communication step in Line~\ref{line:dsmt_ycom}.}


\begin{remark}
    We compare DSMT method with the DSGT method. It is worth noting that DSMT necessitates additional information of $\lambda$ compared to DSGT. This is common for constructing accelerated or optimal methods including those in \cite{lu2021optimal,yuan2022revisiting,scaman2017optimal}. Regarding the communication cost per iteration, each agent in the DSMT method communicates variables $(x_{i,k}, y_{i,k})_{k\geq 0}$, which aligns with the communication load of DSGT. The additional cost primarily stems from the storage of variables $(x_{i,k}^l,y_{i,k}^l, z_{i,k})_{k\geq 0}$. 
\end{remark}

\begin{remark}
    \label{rem:z_minus}
    In the analysis, we introduce the definition $\z_{-1}:= \0$. It is worth noting that the relation $\z_{0} = \beta\z_{-1} + (1-\beta)\g_0$ remains valid due to the initialization $\z_0 = (1-\beta)\g_0$ in Algorithm \ref{alg:dsmt}. As a consequence, the subsequent analysis involving $\z_{-1}$ is valid.
\end{remark}

\section{Preliminary Analysis}
\label{sec:pre}

In this section, we present several preliminary results. We start with establishing the recursions involving the averaged variables $\bar{x}_{k}$, $\bar{y}_k$, and $\bar{z}_k$.

\begin{lemma}
    \label{lem:avg}
    Let Assumption \ref{as:graph} hold and $\y_0 = \z_0 = (1-\beta)\g_0$, we have for any $k\geq 0$ that
    \begin{align*}
        \bar{x}_{k + 1} &= \bar{x}_k - \alpha \bar{y}_k,\\
        \bar{z}_{k + 1} &= \beta\bar{z}_k + (1-\beta)\bar{g}_{k+1},\\
        \bar{y}_{k + 1} &= \bar{z}_{k + 1}.
    \end{align*}
\end{lemma}

\begin{proof}
    See Appendix \ref{app:avg}.
\end{proof}

As highlighted in Remark \ref{rem:sgdm}, the update of $\bar{x}_k$ can be viewed as an approximate version of the SGDM method.
Such an observation prompts us to investigate the behavior of a series of auxiliary variables $\crki{\bar{d}_k}$ defined as follows:
\begin{equation}
    \label{eq:dk}
    \begin{aligned}
        \bar{d}_k:= \begin{cases}
            \bar{x}_k, & k = 0,\\
            \frac{1}{1-\beta}\bar{x}_k - \frac{\beta}{1-\beta}\bar{x}_{k - 1}, & k\geq 1.
        \end{cases}
    \end{aligned}
\end{equation}

The behavior of $\bar{d}_k$ is more like the update of SGD compared to that of $\bar{x}_k$; see Lemma \ref{lem:dk} below. Similar arguments have been explored in prior works, including \cite{yan2018unified,ghadimi2015global,liu2020improved}.

\begin{lemma}
    \label{lem:dk}
    Let Assumption \ref{as:graph} hold. Define $\bar{d}_k$ as in \eqref{eq:dk}, we have $\bar{d}_{k + 1} = \bar{d}_k - \alpha \bar{g}_k$, $k\geq 0$,
    and 
    \begin{align}
        \label{eq:dk_xk}
       \bar{d}_{k} - \bar{x}_k = -\frac{\alpha\beta}{1-\beta}\bar{z}_{k-1},\; k\geq 0.
    \end{align}
\end{lemma}

\begin{proof}
    See Appendix \ref{app:dk}.
\end{proof}

Next, we consider the relation between $\E\brki{f(\bar{d}_k) - f^*}$ and $\E\brki{f(\bar{x}_k) - f^*}$ in light of \eqref{eq:dk_xk} in Lemma \ref{lem:dk}. As a consequence, the follow-up analysis needs not involve the term $\E\brki{f(\bar{x}_k) - f^*}$ which may appear due to Assumption \ref{as:abc}.

\begin{lemma}
    \label{lem:fx_fd}
    Let Assumptions \ref{as:smooth}-\ref{as:graph} hold. Set the stepsize $\alpha$ to satisfy $\alpha\leq (1-\beta)/(3\beta L)$. We have 
    \begin{align*}
        \E\brk{f(\bar{x}_k) - f^*} &\leq \E\brk{f(\bar{d}_k) - f^*} + \frac{\alpha\beta}{1-\beta}\E\brk{\norm{\nabla f(\bar{x}_k)}^2} + \frac{\alpha\beta}{1-\beta}\E\brk{\norm{\bar{z}_{k-1}}^2}.
    \end{align*}
\end{lemma}

\begin{proof}
    See Appendix \ref{app:fx_fd}.
\end{proof}

Our primary goal in this section is to locate and analyze a Lyapunov function to facilitate the convergence analysis. 
We start with Lemma \ref{lem:descent_dk} that presents the recursion of $\E\brki{f(\bar{d}_k) - f^*}$ in light of Assumption \ref{as:smooth} and Lemma \ref{lem:dk}.

\begin{lemma}
    \label{lem:descent_dk}
    Let Assumptions \ref{as:smooth}-\ref{as:graph} hold. Set the stepsize $\alpha$ to satisfy $\alpha \leq \min\crki{1/(4L), 1/(2C)}$. 
    We have 
    \begin{align*}
        &\E\brk{f(\bar{d}_{k + 1}) - f^*} \leq \prt{1 + \frac{2\alpha^2 CL}{n}}\E\brk{f(\bar{d}_k) - f^*} + \frac{\alpha^3 \beta^2 L(L+2C)}{(1-\beta)^2}\E\brk{\norm{\bar{z}_{k-1}}^2}\\
        &\quad - \frac{\alpha}{2}\prt{1 - \frac{4\alpha^2 CL}{n(1-\beta)}}\E\brk{\norm{\nabla f(\bar{x}_k)}^2}+ \frac{\alpha L^2}{n}\E\brk{\norm{\Pi\x_k}^2} + \frac{\alpha^2 L(2C \sigfn + \sigma^2)}{n},
    \end{align*}
    where $\sigfn:= f^* - \frac{1}{n}\sumn f_i^*$.
\end{lemma}

\begin{proof}
    See Appendix \ref{app:descent_dk}.
\end{proof}

Lemma \ref{lem:descent_dk} inspires us to consider the recursion of $\E\brki{\normi{\bar{z}_k}^2}$ as follows.

\begin{lemma}
    \label{lem:zk}
    Let Assumptions \ref{as:smooth}-\ref{as:graph} hold. Set the stepsize $\alpha$ to satisfy $\alpha\leq \min\crki{(1-\beta)/(2\beta L), 1/(4C\beta)}$. We have 
    \begin{align*}
        &\E\brk{\norm{\bar{z}_k}^2}
        \leq \frac{1 + \beta}{2}\E\brk{\norm{\bar{z}_{k-1}}^2} + \frac{2C(1-\beta)^2 }{n}\E\brk{f(\bar{d}_k)-f^*}\\
        &\quad + 3(1-\beta)\E\brk{\norm{\nabla f(\bar{x}_k)}^2}+ \frac{(1-\beta)^2(2C\sigfn + \sigma^2)}{n} \\
        &\quad + \frac{2(1-\beta)L}{n}\prt{L + \frac{C(1-\beta)}{n}}\E\brk{\norm{\Pi\x_k}^2}.
    \end{align*}
\end{lemma}

\begin{proof}
    See Appendix \ref{app:zk}.
\end{proof}

We next consider the recursion of $\E\brki{\normi{\z_k - \nabla F(\Barx{k})}^2}$. 
\begin{lemma}
    \label{lem:zk_nf}
    Let Assumptions \ref{as:smooth}-\ref{as:graph} hold. Set the stepsize $\alpha$ to satisfy $\alpha\leq (1-\beta)/(3L)$. We have for any $k\geq 0$ that
    \begin{align*}
        &\E\brk{\norm{\z_k - \nabla F(\Barx{k})}^2}
       \leq \beta\E\brk{\norm{\z_{k - 1} - \nabla F(\Barx{k-1})}^2} + 2\alpha n(L+C)\E\brk{\norm{\bar{z}_{k - 1}}^2}\\
       &\quad + 2(1-\beta)L(L+C)\E\brk{\norm{\Pi\x_k}^2} + 2n(1-\beta)^2C\E\brk{f(\bar{d}_k) - f^*}\\
       &\quad + 2\alpha n(1-\beta)C\E\brk{\norm{\nabla f(\bar{x}_k)}^2} + n(1-\beta)^2(2C\sigfn + \sigma^2),
    \end{align*} 
    where $\x_{-1}:= \x_0$ and $\z_{-1}=\0$.
\end{lemma}

\begin{proof}
    See Appendix \ref{app:zk_cons}.
\end{proof}

Next, we consider the consensus error terms $\E\brki{\normi{\Pi\x_k}^2}$ and $\E\brki{\normi{\Pi\y_k}^2}$ similar to those in \cite{song2021optimal}.
One of the key steps in the derivation of Lemma \ref{lem:cons_R} involves identifying the error term $\E\brki{\normi{\z_k - \nabla F(\Barx{k})}^2}$, as discussed in Subsection \ref{subsec:motivation}. This particular step is essential for preserving the benefits of the gradient tracking technique, which reduces the impact of data heterogeneity.

\begin{lemma}
    \label{lem:cons_R}
    Let Assumptions \ref{as:smooth}-\ref{as:graph} hold. Set the stepsize $\alpha$ and the parameter $\beta$ to satisfy 
    \begin{align*}
        \alpha\leq\min\crk{\frac{1}{2(L+C)}, \sqrt{\frac{1-\trw}{6c_0^2L(L+C)}}, \sqrt{\frac{1-\beta}{4 CL}}},\; \qquad 1>\beta\geq \trw.
    \end{align*} 
    Then, there exist sequences $\crki{\cR_k^x}$ and $\crki{\cR_k^y}$ such that 
    \begin{align*}
        &\E\brk{\norm{\Pi\x_k}^2}\leq \E\brk{\norm{\tpx_k}^2}\leq \cR_k^x,\\
        & \E\brk{\norm{\Pi\y_k}^2}\leq \E\brk{\norm{\tpy_k}^2}\leq \cR_k^y,
    \end{align*}
    and 
    \begin{align}
        \cR_{k + 1}^x &\leq \trw\cR_k^x + \frac{\alpha^2 c_0\trw}{1-\trw}\cR_k^y,\label{eq:Rx_re}\\
        \cR_{k+1}^y &\leq \frac{1+\trw}{2}\cR_k^y + \frac{15c_0(1-\beta)^2L(L+C)}{1-\trw}\cR_k^x +9\alpha n c_0(1-\beta)(C+L)\E\brk{\norm{\bar{z}_{k - 1}}^2}\nonumber\\
        &\quad + \frac{3c_0(1-\beta)^2\beta}{1-\trw}\E\brk{\norm{\z_{k - 1} - \nabla F(\Barx{k-1})}^2} + 12nc_0(1-\beta)^2C\E\brk{f(\bar{d}_k) - f^*}\nonumber\\
        &\quad + 20\alpha n c_0(1-\beta)(C+L)\E\brk{\norm{\nabla f(\bar{x}_k)}^2}+ 6nc_0(1-\beta)^2(2C\sigfn + \sigma^2),\label{eq:Ry_re}
    \end{align}
    where $c_0 = 14$ is defined in Lemma \ref{lem:lca} and
    \begin{equation}
        \label{eq:cR0}
        \begin{aligned}
            \cR_0^x &\leq c_0\norm{\Pi\x_0}^2 + \frac{\alpha^2 c_0\trw}{1-\trw}\E\brk{\norm{\Pi\y_0}^2},\\
            \cR_0^y &:= c_0\E\brk{\norm{\Pi\y_0}^2} + nc_0\trw(1-\beta)^2(2C\sigfn + \sigma^2).
        \end{aligned}
    \end{equation}
\end{lemma}

\begin{proof}
    See Appendix \ref{app:cons_R}.
\end{proof}

We are now ready to introduce the Lyapunov function $\cL_k$ defined as follows:
\begin{equation}
    \label{eq:lya0}
    \begin{aligned}
        \cL_k&:=\E\brk{f(\bar{d}_k) - f^*} + \frac{3\alpha L^2}{n(1-\trw)} \cR_k^x + \frac{12\alpha^3c_0\trw L^2}{n(1-\trw)^3} \cR_k^y\\
        &\quad + \frac{4\alpha^3 L(L+2C)}{(1-\beta)^3}\E\brk{\norm{\bar{z}_{k-1}}^2} + \frac{48\alpha^3c_0^2\trw L^2}{n(1-\trw)^3}\E\brk{\norm{\z_{k-1} - \nabla F(\Barx{k-1})}^2}.
    \end{aligned}
\end{equation}
Lemma \ref{lem:lya} below presents the recursion for $\cL_k$ by combining the results in Lemmas \ref{lem:descent_dk}-\ref{lem:cons_R}.

\begin{lemma}
    \label{lem:lya}
    Let Assumptions \ref{as:smooth}-\ref{as:graph} hold. Set the stepsize $\alpha$ and the parameter $\beta$ to satisfy 
    \begin{align*}
        &  \alpha\leq\min\crk{\sqrt{\frac{(1-\trw)^2}{240c_0^2L(L+C)}},\frac{1-\beta}{24 (L+C)}, \prt{\frac{(1-\trw)^3}{4032 c_0^2(1-\beta)L^2(C+L)}}^{1/3}},\\ 
        & 1 > \beta\geq\trw.
    \end{align*} 
    We have 
    \begin{align*}
        &\cL_{k+1}\leq \prt{1 + \frac{4\alpha^2 CL}{n} + \frac{240\alpha^3c_0^2\trw CL^2(1-\beta)^2}{(1-\trw)^3}}\cL_k - \frac{\alpha}{4}\E\brk{\norm{\nabla f(\bar{x}_k)}^2}\\
        &\quad + \frac{\alpha^2 L(2C\sigfn + \sigma^2)}{n}+ \brk{\frac{72c_0^2\trw L^2(1-\beta)^2}{(1-\trw)^3} + \frac{4L(L+2C)}{n(1-\beta)}}\alpha^3(2C\sigfn + \sigma^2).
    \end{align*}

    Moreover, denote $\Delta_0:= f(\bar{x}_0) - f^*$, the term $\cL_0$ can be upper bounded by 
    \begin{equation}
        \label{eq:cL0_bound}
        \begin{aligned}
            \cL_0 &= \Delta_0 + \frac{3\alpha L^2}{n(1-\trw)}\cR_0^x + \frac{12\alpha^3c_0\trw L^2}{n(1-\trw)^3}\cR_0^y+ \frac{48\alpha^3c_0^2\trw L^2}{n(1-\trw)^3}\norm{\nabla F(\x_0)}^2\\
            &\leq 2\Delta_0 + \frac{9\alpha L^2\norm{\Pi\x_0}^2}{n(1-\trw)} + \frac{27\alpha^3 c_0^2L^2(1-\beta)^2(2C\sigfn + \sigma^2)}{(1-\trw)^2}\\
            &\quad + \frac{72\alpha^3c_0^2 L^2\sumn\norm{\nabla f_i(\bar{x}_0)}^2}{n(1-\trw)^3}.
        \end{aligned}
    \end{equation}
\end{lemma}

\begin{proof}
    See Appendix \ref{app:lya}.
\end{proof}

In summary, Lemma \ref{lem:descent_dk} inspires us to derive the recursions of the corresponding error terms, outlined in Lemmas \ref{lem:zk}-\ref{lem:cons_R}, and to consider the Lyapunov function in \eqref{eq:lya0}. Moreover, due to the fact that $\normi{\tpi\tW}_2$ may greater than one, the consensus error terms $\E\brki{\normi{\Pi\x_k}^2}$ and $\E\brki{\normi{\Pi\y_k}^{2}}$ are proved to be $R$-linear sequences with ``additional errors'' similar as those in \cite{song2021optimal}. Notably, the analysis of DSMT differs from \cite{song2021optimal} mainly due to the need for dealing with the stochastic gradients under the general variance condition \eqref{eq:abc} and the momentum variable $\z_k$.

\section{Main Results}
\label{sec:main}

In this section, we present the main results, which outline the convergence rates and the transient times of the DSMT method for minimizing smooth objective functions with or without the PL condition. These results are detailed in Subsections \ref{subsec:ncvx} and \ref{subsec:pl}, respectively. The formal definitions of transient times $\KTN$ and $\KTP$ under general smooth objective functions and smooth objective functions satisfying the PL condition are defined, as follows:

\begin{subequations}
    \begin{align}
        &\KTN:=\inf_{K}\crk{\min_{t=0,1,\ldots, k-1}\E\brk{\norm{\nabla f(\bar{x}_t)}^2}\leq \order{\frac{1}{\sqrt{nk}}},\ \forall k\geq K}\label{eq:KTN},\\
        &\KTP:= \inf_{K}\crk{\frac{1}{n}\sumn \E\brk{f(x_{i,k})-f^*}\leq \order{\frac{1}{nk}},\ \forall k\geq K }\label{eq:KTP}.
    \end{align}
\end{subequations}
Note that $\orderi{1/\sqrt{nk}}$ and $\orderi{1/(nk)}$ represent the convergence rates of the centralized stochastic gradient method under the two considered conditions, respectively.

\subsection{General Nonconvex Case}
\label{subsec:ncvx}

We begin by presenting the convergence result of DSMT for minimizing smooth nonconvex objective functions, as demonstrated in Theorem \ref{thm:ncvx} below.

\begin{theorem}
    \label{thm:ncvx}
    Let Assumptions \ref{as:smooth}-\ref{as:graph} hold. {For a given $K\geq 1$,} let the stepsize $\alpha$ and the parameter $\beta$ to satisfy 
    \begin{align*}
        \alpha\leq\min&\left\{\sqrt{\frac{(1-\trw)^2}{240c_0^2L(L+C)}},\frac{1-\beta}{24 (L+C)}, \prt{\frac{(1-\trw)^3}{4032 c_0^2(1-\beta)L^2(C+L)}}^{1/3},\right.\\
        &\quad \left. \sqrt{\frac{n}{8CLK}}, \prt{\frac{(1-\trw)^3}{480c_0^2CL^2(1-\beta)^2K}}^{1/3} \right\},\qquad 1 > \beta\geq\trw.
    \end{align*} 
    Then, we have 
    \begin{equation}
        \label{eq:minE}
        \begin{aligned}
            &\min_{k=0,1,\cdots, K-1} \E\brk{\norm{\nabla f(\bar{x}_k)}^2} \leq \frac{24\Delta_0}{\alpha K} + \frac{108 L^2\norm{\Pi\x_0}^2}{n(1-\trw)K} + \frac{4\alpha L(2C\sigfn + \sigma^2) }{n}\\
            &\quad + \frac{324c_0^2\alpha^2 L^2(1-\beta)^2(2C\sigfn + \sigma^2)}{(1-\trw)^2 K}+ \frac{864c_0^2\alpha^2L^2\sumn\norm{\nabla f_i(\bar{x}_0)}^2}{n(1-\trw)^3 K}\\
            &\quad + \brk{\frac{L(L+2C)}{n(1-\beta)} + \frac{18c_0^2\trw L^2(1-\beta)^2}{(1-\trw)^3}}16\alpha^2(2C\sigfn + \sigma^2).
        \end{aligned}
    \end{equation}

    In addition, if we set 
    \begin{equation}
        \label{eq:alpha_ncvx}
        \begin{aligned}
            \alpha = \frac{1}{\sqrt{\frac{L(2C\sigfn + \sigma^2) K}{3n\Delta_0}} + \gamma},\ \quad \beta = 1 - \frac{1}{n^{1/3}}\prt{1 - \trw},
        \end{aligned}
    \end{equation}
    where 
    \begin{align*}
        \gamma:= \prt{\frac{4032c_0^2 L^2(C + L)}{n^{1/3}(1-\trw)^2}}^{1/3} + \frac{24 c_0n^{1/3}(L+C)}{1-\trw} + \sqrt{\frac{8CL K}{n}} + \prt{\frac{480c_0^2 CL^2 K }{n^{2/3}(1-\trw)}}^{1/3},
    \end{align*}
    then 
    \begin{equation}
    \label{eq:minE_order}
        \begin{aligned}
            &\min_{k=0,1,\cdots, K-1} \E\brk{\norm{\nabla f(\bar{x}_k)}^2} =\order{\sqrt{\frac{\Delta_0 L(2C\sigfn + \sigma^2)}{n K}} + \sqrt{\frac{\Delta_0^2 CL}{nK}} \right.\\
            &\left.\quad + \frac{\Delta_0 (CL^2)^{1/3}}{n^{2/9}(1-\trw)^{1/3}K^{2/3}}+ \frac{n^{1/3}\Delta_0(L+C) +\norm{\Pi\x_0}^2L^2/n}{(1-\trw)K}\right.\\
            &\left.\quad + \frac{n^{1/3}\Delta_0 L}{K^2} + \frac{\Delta_0 L \sumn\normi{\nabla f_i(\bar{x}_0)}^2}{(1-\trw)^3(2C\sigfn + \sigma^2)K^2}},
        \end{aligned}
    \end{equation}
    where $\orderi{\cdot}$ hides the numerical constants.

\end{theorem}

\begin{proof}
    We first state Lemma \ref{lem:hhelp} from \cite[Lemma 6]{mishchenko2020random} that provides a direct link connecting Lemma \ref{lem:lya} to the desired results.

        \begin{lemma}
        \label{lem:hhelp}
        If there exist constants $a,b,c \geq 0$ and non-negative sequences $\crki{s_t},\crki{q_t}$  such that for any $t$, we have $s_{t+1} \leq (1+a)s_{t}-bq_t+c$,
        then it holds that
        \begin{align*}
            \min\limits_{t=0,1,...,T-1} q_t \leq \frac{(1+a)^T}{bT}s_0+\frac{c}{b}.
        \end{align*}
        \end{lemma}
        Applying Lemma \ref{lem:hhelp} to Lemma \ref{lem:lya} leads to 
        \begin{equation}
            \label{eq:minEK}
            \begin{aligned}
                &\min_{k=0,1,\cdots, K-1} \E\brk{\norm{\nabla f(\bar{x}_k)}^2} \leq \frac{4\prt{1 + \frac{4\alpha^2 CL}{n} + \frac{240\alpha^3c_0^2\trw CL^2(1-\beta)^2}{(1-\trw)^3}}^{K}}{\alpha K}\cL_0\\
                &\quad + \frac{4\alpha L(2C\sigfn + \sigma^2) }{n}+ \brk{\frac{L(L+2C)}{n(1-\beta)} + \frac{18c_0^2\trw L^2(1-\beta)^2}{(1-\trw)^3}}16\alpha^2(2C\sigfn + \sigma^2).
            \end{aligned}
        \end{equation}
Letting 
        \begin{align*}
            \alpha \leq \min\crk{\sqrt{\frac{n}{8CLK}}, \prt{\frac{(1-\trw)^3}{480c_0^2CL^2(1-\beta)^2 K}}^{1/3} }
        \end{align*}
        yields 
        \begin{equation}
            \label{eq:quasi}
            \begin{aligned}
                &\prt{1 + \frac{4\alpha^2 CL}{n} + \frac{240\alpha^3c_0^2\trw CL^2(1-\beta)^2}{(1-\trw)^3}}^K \\
                &\leq \exp\crk{\brk{\frac{4\alpha^2 CL}{n} + \frac{240\alpha^3c_0^2\trw CL^2(1-\beta)^2}{(1-\trw)^3}}K}\leq \exp(1).
            \end{aligned}
        \end{equation}
Substituting \eqref{eq:quasi} and the expression of $\cL_0$ in \eqref{eq:cL0_bound} into \eqref{eq:minEK} yields \eqref{eq:minE}.
        Set $\alpha$ as in \eqref{eq:alpha_ncvx}, then 
        \begin{equation}
            \label{eq:alpha_I}
            \begin{aligned}
                &\frac{\Delta_0}{\alpha K} = \sqrt{\frac{\Delta_0 L(2C\sigfn + \sigma^2)}{3n K}} + \Delta_0\sqrt{\frac{8CL}{nK}} + \frac{\Delta_0(480c_0^2 CL^2)^{1/3}}{n^{2/9}(1-\trw)^{1/3}K^{2/3}} \\
                &\quad + \frac{\Delta_0[4032c_0^2 L^2(C+L)]^{1/3}}{n^{1/9}(1-\trw)^{2/3}K} + \frac{24 c_0 \Delta_0n^{1/3}(L+C)}{(1-\trw) K},\\
                &\frac{4\alpha L(2C\sigfn + \sigma^2)}{n} \leq  4\sqrt{\frac{3L(2C\sigfn+\sigma^2)\Delta_0}{nK}},\ \alpha^2 \leq \frac{3n\Delta_0}{L(2C\sigfn + \sigma^2)K}.
            \end{aligned}
        \end{equation}
Substituting \eqref{eq:alpha_I} into \eqref{eq:minE} yields the desired result \eqref{eq:minE_order}.
\end{proof}

Subsequently, we compute the transient time required for DSMT to achieve comparable performance to the centralized SGD method when minimizing smooth objective functions, as detailed in Corollary \ref{cor:ncvx} below.
\begin{corollary}
    \label{cor:ncvx}
    Let Assumptions \ref{as:smooth}-\ref{as:graph} hold. Set the stepsize $\alpha$ and the parameter $\beta$ to satisfy \eqref{eq:alpha_ncvx}. Then the transient time $\KTN$ required for DSMT method to achieve comparable performance as the centralized SGD method behaves as 
    \begin{equation}
        \label{eq:KTN_full}
        \begin{aligned}
            \KTN = \max&\crk{ \order{\frac{n^{5/3}}{1-\lambda}}, \order{\frac{\norm{\Pi\x_0}^4}{(1-\lambda) n}},\right.\\
            &\left.\quad \order{\frac{n^{1/3}\prt{\sumn\norm{\nabla f_i(\bar{x}_0)}^2}^{2/3}}{1-\lambda}}}.
        \end{aligned}
    \end{equation}
    In addition, if we initialize all agents at the same {initial point} and assume that $\sumn\normi{\nabla f_i(\bar{x}_0)}^2= \orderi{n^2}$, then the transient time $\KTN$ stated in \eqref{eq:KTN_full} further reduces to 
    \begin{align*}
        \KTN = \order{\frac{n^{5/3}}{1-\lambda}}.
    \end{align*}
\end{corollary}

\begin{remark}
    \label{rem:mild_KT}
    It is worth noting that initializing all the agents at the same {initial point} and assuming $\sumn\normi{\nabla f_i(\bar{x}_0)}^2= \orderi{n^2}$ are mild conditions. 
    In practice, the term $\sumn\normi{\nabla f_i(\bar{x}_0)}^2$ often behaves as $\orderi{n}$ assuming $\bar{x}_0$ is fixed.
\end{remark}

\begin{proof}
    We hide the constants that are independent of $n$ and $(1-\lambda)$ in the following. From \eqref{eq:minE_order}, we have 
    \begin{equation}
    \label{eq:minE_order1}
        \begin{aligned}
        &\min_{k=0,1,\cdots, K-1} \E\brk{\norm{\nabla f(\bar{x}_k)}^2} =\order{\frac{1}{\sqrt{nK}}} \order{1 + \frac{n^{5/(18)}}{(1-\trw)^{1/3}K^{1/6}} + \frac{n^{5/6}}{(1-\trw)\sqrt{K}}\right.\\
        &\quad \left.+ \frac{\norm{\Pi\x_0}^2}{(1-\trw)\sqrt{nK}} + \frac{n^{5/6}}{K^{3/2}} + \frac{\sqrt{n}\sumn\normi{\nabla f_i(\bar{x}_0)}^2}{(1-\trw)^3K^{3/2}}}.
    \end{aligned}
    \end{equation}
Noting that 
    \begin{equation}
        \label{eq:trw}
        \begin{aligned}
            1-\trw &= 1 - \sqrt{\frac{1}{1+\sqrt{1-\lambda^2}}} = \frac{\sqrt{1-\lambda^2}}{\sqrt{1 + \sqrt{1-\lambda^2}}\prt{1 + \sqrt{1 + \sqrt{1 - \lambda^2}}}}\\
            &\sim \order{\sqrt{1-\lambda}}.
        \end{aligned}
    \end{equation}
    According to the definition for transient times in \eqref{eq:KTN}, we obtain the desired result \eqref{eq:KTN_full} by invoking \eqref{eq:minE_order1} and \eqref{eq:trw}.

\end{proof}

\subsection{PL Condition}
\label{subsec:pl}
    In this part, we will consider a specific nonconvex condition known as the Polyak-{\L}ojasiewicz (PL) condition in Assumption \ref{as:PL}. Overparameterized models often satisfy this condition. Notably, the strong convexity condition implies the PL condition \cite{karimi2016linear}.
            
    \begin{assumption}
        \label{as:PL}
        There {exists} $\mu>0$, such that the aggregate function $f(x)=\frac{1}{n}\sum_{i=1}^n f_i(x)$ satisfies
        \begin{align}
            2\mu\prt{f(x)-f^*} \leq \norm{\nabla f(x)}^2,
        \end{align}
        for all $x \in \R^p$, where $f^*:= \inf_{x\in\R^p}f(x)$.
    \end{assumption}
    
    In light of Assumption \ref{as:PL}, we construct a ``contractive'' recursion {for the quantity} $\E\brki{f(\bar{d}_k) - f^*}$.

    \begin{lemma}
        \label{lem:contract_dk}
        Let Assumptions \ref{as:smooth}-\ref{as:graph} and \ref{as:PL} hold, and let $K>0$. Set $\bar{d}_k$ as in \eqref{eq:dk} and let the stepsize $\alpha$ to satisfy $\alpha \leq \min\crki{\mu/(4LC), (1-\beta)/(2L)}$. 
        We have 
        \begin{align*}
            \E\brk{f(\bar{d}_{k + 1}) - f^*} &\leq \prt{1 - \frac{\alpha\mu}{2}}\E\brk{f(\bar{d}_k) - f^*} + \frac{\alpha L^2}{n}\E\brk{\norm{\Pi\x_k}^2} + \frac{\alpha^2 L(2C\sigfn + \sigma^2)}{n}\\
            &\quad + \frac{\alpha^3 \beta L(L+C)}{(1-\beta)^2}\E\brk{\norm{\bar{z}_{k-1}}^2}.
        \end{align*}
    \end{lemma}
    
    \begin{proof}
        See Appendix \ref{app:contract_dk}.
    \end{proof}

    Thanks to Assumption \ref{as:PL}, we are now able to show that $\cL_k$ is a $Q$-linear sequence with ``additional errors'' in Lemma \ref{lem:cL_pl}. The derivation is similar to those in the proof of Lemma \ref{lem:lya}. 

    \begin{lemma}
        \label{lem:cL_pl}
        Let Assumptions \ref{as:smooth}-\ref{as:graph} and \ref{as:PL} hold, and let $k\geq0$. Set the stepsize $\alpha$ and the parameter $\beta$ to satisfy 
        \begin{align*}
            \alpha\leq \min&\crk{\frac{1-\beta}{3\mu}, \frac{1-\beta}{486c_0^2(L+C)}, \sqrt{\frac{(1-\trw)^2}{240c_0^2L(L+C)}}, \sqrt{\frac{1-\trw}{324c_0^2 (L+C)^2}},\right.\\
            &\left.\quad \sqrt{\frac{\mu(1-\beta)^2}{960c_0^2 L(L+C)^2}}},\qquad 1\geq \beta\geq 1 - \frac{11(1-\trw)}{12}.
        \end{align*}
        We have 
        \begin{equation}
            \label{eq:cL_pl}
            \begin{aligned}
                \cL_{k + 1} &\leq \prt{1 - \frac{\alpha\mu}{4}}\cL_k\\
                &\quad +\brk{\frac{L}{n} + \frac{72\alpha c_0^2\trw L^2(1-\beta)^2}{(1-\trw)^3} + \frac{4\alpha L(L+2C)}{n(1-\beta)}}\alpha^2(2C\sigfn + \sigma^2).
               \end{aligned}
        \end{equation}
    \end{lemma}

    \begin{proof}
        See Appendix \ref{app:cL_pl}.
    \end{proof}

    We can obtain the convergence rate of DSMT in Theorem \ref{thm:pl} by noting another Lyapunov function $\cH_k$:
    \begin{equation}
        \label{eq:lya_hk}
        \begin{aligned}
            \cH_k&:= \cR_k^x + \frac{3\alpha^2 c_0\trw}{(1-\trw)^2} \cR_k^y + \frac{405\alpha^3 nc_0^2(C+L)}{(1-\trw)^3} \E\brk{\norm{\bar{z}_{k-1}}^2}\\
            &\quad + \frac{54\alpha^2c_0^2\trw}{(1-\trw)^2}\E\brk{\norm{\z_{k-1} -\nabla F(\Barx{k-1})}^2}.
        \end{aligned}
    \end{equation}    

    \begin{theorem}
        \label{thm:pl}
        Let Assumptions \ref{as:smooth}-\ref{as:graph} and \ref{as:PL} hold. Set the stepsize $\alpha$ and the parameter $\beta$ to satisfy 
        \begin{align}
            &\alpha\leq \min\crk{\frac{1-\beta}{3\mu}, \frac{1-\beta}{486c_0^2(L+C)}, \sqrt{\frac{(1-\trw)^2}{240c_0^2L(L+C)}}, \sqrt{\frac{1-\trw}{324c_0^2 (L+C)^2}},\right.\nonumber\\
            &\left.\quad \sqrt{\frac{\mu(1-\beta)^2}{960c_0^2 L(L+C)^2}}},\label{eq:alpha_pl}\\
            &1\geq \beta\geq 1 - \frac{11(1-\trw)}{12}.
        \end{align} 
        We have 
        \begin{align*}
            &\frac{1}{n}\sumn\E\brk{f(x_{i,k}) - f^*} \leq 33\cL_0\prt{1 - \frac{\alpha\mu}{4}}^k + \frac{5L\cH_0}{n}\prt{\frac{5+\beta}{6}}^k + \frac{20\alpha L (2C\sigfn + \sigma^2)}{n\mu}\\
            &\quad + \brk{\frac{48 L\cS_1}{\mu} + \frac{504c_0^2(1+\cS_2)(1-\beta)}{1-\trw}+ \frac{32(6L+C)(1-\trw)}{n\mu(1-\beta)}}\frac{4\alpha^2 L(2C\sigfn + \sigma^2)}{1-\trw},
        \end{align*}
        where 
        \begin{align*}
            \cS_1&:= \frac{18c_0^2(1-\beta)^2}{(1-\trw)^2} + \frac{(L+2C)(1-\trw)}{nL (1-\beta)},\; \cS_2:= \frac{L}{n\mu} + \frac{L\cS_1}{\mu},
        \end{align*}
        and 
        \begin{equation}
        \label{eq:cH0_bound}
            \begin{aligned}
            \cH_0 &= \cR_0^x + \frac{3\alpha^2 c_0\trw}{(1-\trw)^2}\cR_0^y + \frac{54\alpha^2c_0^2\trw}{(1-\trw)^2}\norm{\nabla F(\Barx{0})}^2\\
            &\leq 2c_0\norm{\Pi\x_0}^2 + \frac{12\alpha^2 nc_0^2(1-\beta)^2(C\Delta_0 + 2C\sigfn + \sigma^2)}{(1-\trw)^2} + \frac{72 \alpha^2 c_0^2 \sumn\norm{\nabla f_i(\bar{x}_0)}^2}{(1-\trw)^2}.
        \end{aligned}
        \end{equation}

    \end{theorem}

    \begin{proof}
        See Appendix \ref{app:pl}.
    \end{proof}

    Based on Theorem \ref{thm:pl}, we can first compare the non-exponentially decreasing terms in the error bounds for different distributed stochastic gradient methods under a constant stepsize $\alpha$; see Table \ref{tab:comp_pl}. It can be seen that the DSMT method exhibits the lowest static error when compared to the state-of-art methods by choosing $\beta = \order{1-\sqrt{(1-\lambda)/n}}$, $n\geq 2$.
    
    \begin{table}[h]
    \setlength{\tabcolsep}{80pt}
        \centering
        \begin{tabular}{@{}cc@{}}
        \toprule
        Method                           & Static Error                                                                                       \\ \midrule
        DSGD \cite{pu2021sharp}          & $\order{\frac{\alpha}{n} + \frac{\alpha^{3/2}}{1-\lambda} + \frac{\alpha^2}{(1-\lambda)^2}}$ \\
        ED \cite{alghunaim2021unified}   & $\order{\frac{\alpha}{n} + \frac{\alpha^2}{1-\lambda} + \frac{\alpha^4}{n(1-\lambda)^3}}$    \\
        DSGT \cite{alghunaim2021unified} & $\order{\frac{\alpha}{n} + \frac{\alpha^2}{1-\lambda} + \frac{\alpha^4}{n(1-\lambda)^4}}$    \\ \midrule
        DSMT                             & $\order{\frac{\alpha}{n} + \brk{\frac{1-\beta}{1-\lambda} + \frac{1}{n(1-\beta)}}\alpha^2}$                               \\ \midrule 
        \makecell[c]{DSMT\\ $\beta = \orderi{1-\sqrt{(1-\lambda)/n}}$}                            & $\order{\frac{\alpha}{n} + \frac{\alpha^2}{\sqrt{n(1-\lambda)}}}$                               \\ 
        \bottomrule
        \end{tabular}
        \caption{Comparison of different methods regarding the non-exponentially decreasing terms for minimizing smooth objective functions satisfying the PL condition.}
        \label{tab:comp_pl}
    \end{table}

    Table \ref{tab:comp_pl} also shows that proper choice of $\beta$ improves the final error of DSMT, demonstrating the benefits of momentum acceleration in decentralized optimization methods.

    By considering specific stepsize $\alpha$ and parameter $\beta$, we can write down the following convergence result of DSMT based on Theorem \ref{thm:pl}.
    \begin{corollary}
        \label{cor:order_pl}
        Let the conditions in Theorem \ref{thm:pl} hold. Furthermore, let the stepsize $\alpha$ and parameter $\beta$ satisfy 
        \begin{align}
            \label{eq:alpha_order}
            \alpha = \alpha_1 = \frac{4}{\mu K}\ln\prt{\frac{33n\mu^2 K \cL_0}{(2C\sigfn + \sigma^2) L}},\ \beta = 1 - \frac{(1-\trw)}{\sqrt{n}}, n\geq 2,
        \end{align}
        {for some given $K\geq 1$.}
        Then, the final iterates $x_{i,K}$ for all $i\in\cN$ satisfy 
        \begin{equation}
            \label{eq:fK_bound}
            \begin{aligned}
                \frac{1}{n}\sumn\E\brk{f(x_{i,K}) - f^*} &\leq 33\cL_0\exp\prt{-\frac{\alpha\mu K}{4}} + \frac{5L\cH_0}{n}\exp\prt{-\frac{(1-\trw) K}{\sqrt{n}}}\\
                &\quad + \torder{\frac{(2C\sigfn + \sigma^2)}{nK} + \frac{(2C\sigfn + \sigma^2)}{\sqrt{n(1-\lambda)} K^2} },
        \end{aligned}
        \end{equation}
        where $\torderi{\cdot}$ hides the term $\ln\prt{\frac{33n\mu^2 K \cL_0}{(2C\sigfn + \sigma^2) L}}$ and some constants that are independent of $n$ and $(1-\lambda)$. In particular, we have 
        \small
        \begin{equation}
            \label{eq:fK_bound_order}
            \begin{aligned}
                &\frac{1}{n}\sumn\E\brk{f(x_{i,K}) - f^*}\\
                & = \tilde{\mathcal{O}}\left\{ \brk{\norm{\Pi\x_0}^2 + \frac{(\Delta_0 + C\sigfn + \sigma^2)}{\mu^2 K^2} + \frac{\sumn\norm{\nabla f_i(\bar{x}_0)}^2L^2}{\mu^2 (1-\trw)^2 K^2 }}\exp\prt{-\frac{(1-\trw) K}{6\sqrt{n}}}\right.\\
                &\left. + \brk{\Delta_0 + \frac{ L^2\norm{\Pi\x_0}^2}{n\mu (1-\trw) K} + \frac{L^2(2C\sigfn + \sigma^2)}{n\mu^3 K^3} + \frac{L^2\sumn\norm{\nabla f_i(\bar{x}_0)}^2}{\mu^3 (1-\trw)^3 K^3}}\exp\prt{-\frac{(1-\trw)\mu K}{4L}}\right.\\
                &\left. + \frac{L(C\sigfn + \sigma^2)}{n\mu^2 K} + \frac{L(C\sigfn + \sigma^2)}{\sqrt{n}(1-\trw)\mu^2 K^2} \right\}.
            \end{aligned}
        \end{equation}\normalsize

    \end{corollary}

    \begin{proof}
        Since $\beta = 1 - (1-\trw)/\sqrt{n}$, the constants $\cS_1$ and $\cS_2$ become
        \begin{align*}
            \cS_1 = \frac{18c_0^2}{n} + \frac{L+2C}{\sqrt{n}L},\; \cS_2 = \frac{L}{n\mu} +\frac{18c_0^2 L}{n\mu} + \frac{L+2C}{\sqrt{n}\mu}.
        \end{align*}
        Denote 
        \begin{align*}
            \cS_3:= \frac{48L}{\mu}\prt{\frac{18c_0^2}{\sqrt{n}} + \frac{L+2C}{L}} + 504c_0^2 (1+\cS_1) + \frac{32(6L+C)}{\mu}\sim\order{1},
        \end{align*}
        we have 
        \begin{equation}
            \label{eq:fK}
            \begin{aligned}
                \frac{1}{n}\sumn\E\brk{f(x_{i,K}) - f^*} &\leq 33\cL_0\exp\prt{-\frac{\alpha\mu K}{4}} + \frac{5L\cH_0}{n}\exp\prt{-\frac{(1-\trw) K}{6\sqrt{n}}}\\
                &\quad \frac{20\alpha L(2C\sigfn + \sigma^2)}{n\mu} + \frac{4\alpha^2\cS_3 L(2C\sigfn + \sigma^2)}{\sqrt{n}(1-\trw)}.
            \end{aligned}
        \end{equation}

        Let $\bar{\alpha}$ satisfy \eqref{eq:alpha_pl}. It suffices to discuss the following two cases:

        \textbf{Case I:} If $\bar{\alpha}< \alpha_1$, then we substitute $\bar{\alpha}$ into \eqref{eq:fK} and obtain 
        \begin{multline}
            \label{eq:small_K}
                \frac{1}{n}\sumn\E\brk{f(x_{i,K}) - f^*} \leq 33\cL_0\exp\prt{-\frac{\bar{\alpha}\mu K}{4}} + \frac{5L\cH_0}{n}\exp\prt{-\frac{(1-\trw) K}{6\sqrt{n}}}\\
                \quad + \frac{80 L(2C\sigfn + \sigma^2)\ln\prt{\frac{33n\mu^2 K \cL_0}{(2C\sigfn + \sigma^2) L}}}{n\mu^2 K} + \frac{64L\cS_3(2C\sigfn + \sigma^2)\ln^2\prt{\frac{33n\mu^2 K \cL_0}{(2C\sigfn + \sigma^2) L}}}{\sqrt{n}(1-\trw)\mu^2 K^2}.
        \end{multline}
In this case, we have $\bar{\alpha} \sim \orderi{(1-\trw)/[\sqrt{n}(L+C)]}$. Then,
        \small
        \begin{equation}
            \label{eq:small_K_order}
            \begin{aligned}
                &\frac{1}{n}\sumn\E\brk{f(x_{i,K}) - f^*} \\
                &= \tilde{\mathcal{O}}\left\{ \brk{\norm{\Pi\x_0}^2 + \frac{(C\Delta_0 + C\sigfn + \sigma^2)}{\mu^2 K^2} + \frac{\sumn\norm{\nabla f_i(\bar{x}_0)}^2}{\mu^2 (1-\trw)^2 K^2 }}\exp\prt{-\frac{(1-\trw) K}{6\sqrt{n}}}\right.\\
                &\left. + \brk{\Delta_0 + \frac{ L^2\norm{\Pi\x_0}^2}{n\mu (1-\trw) K} + \frac{L^2(2C\sigfn + \sigma^2)}{n\mu^3 K^3} + \frac{L^2\sumn\norm{\nabla f_i(\bar{x}_0)}^2}{\mu^3 (1-\trw)^3 K^3}}\exp\prt{-\frac{(1-\trw)\mu K}{4\sqrt{n}(L+C)}}\right.\\
                &\left. + \frac{L(C\sigfn + \sigma^2)}{n\mu^2 K} + \frac{L(C\sigfn + \sigma^2)}{\sqrt{n}(1-\trw)\mu^2 K^2} \right\}.
            \end{aligned}
        \end{equation}\normalsize

        \textbf{Case II:} If $\bar{\alpha}\geq \alpha_1$, then we substitute $\alpha_1$ into \eqref{eq:fK} and obtain 
        \begin{equation}
            \label{eq:large_K}
            \begin{aligned}
                &\frac{1}{n}\sumn\E\brk{f(x_{i,K}) - f^*} \leq \frac{(2C\sigfn + \sigma^2)L}{n\mu^2 K} + \frac{5L\cH_0}{n}\exp\prt{-\frac{(1-\trw) K}{6\sqrt{n}}}\\
                &\quad + \frac{80 L(2C\sigfn + \sigma^2)\ln\prt{\frac{33n\mu^2 K \cL_0}{(2C\sigfn + \sigma^2) L}}}{n\mu^2 K} + \frac{64L\cS_3(2C\sigfn + \sigma^2)\ln^2\prt{\frac{33n\mu^2 K \cL_0}{(2C\sigfn + \sigma^2) L}}}{\sqrt{n}(1-\trw)\mu^2 K^2}.
            \end{aligned}
        \end{equation}
In this case, we have from \eqref{eq:large_K} that 
        \begin{equation}
            \label{eq:large_K_order}
            \begin{aligned}
                &\frac{1}{n}\sumn\E\brk{f(x_{i,K}) - f^*} = \tilde{\mathcal{O}}\left\{\frac{L(C\sigfn + \sigma^2)}{n\mu^2 K} + \frac{L(C\sigfn + \sigma^2)}{\sqrt{n}(1-\trw)\mu^2 K^2}\right.\\
                &\left. + \brk{\norm{\Pi\x_0}^2 + \frac{(\Delta_0 + C\sigfn + \sigma^2)}{\mu^2 K^2} + \frac{\sumn\norm{\nabla f_i(\bar{x}_0)}^2L^2}{\mu^2 (1-\trw)^2 K^2 }}\exp\prt{-\frac{(1-\trw) K}{6\sqrt{n}}} \right\}.
            \end{aligned}
        \end{equation}
Combining \eqref{eq:small_K} and \eqref{eq:large_K} yields the desired result \eqref{eq:fK_bound}. Combining \eqref{eq:small_K_order} and \eqref{eq:large_K_order} leads to the desired result \eqref{eq:fK_bound_order}.
    \end{proof}

    Finally, we are able to derive the transient time of DSMT under smooth objective functions satisfying the PL condition.
    \begin{corollary}
        \label{cor:kt_pl}
        Let the conditions in Corollary \ref{cor:order_pl} hold. Then the transient time $\KTP$ required for DSMT to achieve comparable performance as centralized SGD method when minimizing smooth objective functions satisfying the PL condition behaves as 
        \begin{equation}
            \label{eq:KTP_full}
            \begin{aligned}
                \KTP = \torder{\sqrt{\frac{n}{1-\lambda}}},
            \end{aligned}
        \end{equation}
        where $\torderi{\cdot}$ hides the following logarithm terms and other constants that are independent of $n$ and $(1-\lambda)$:
        \begin{align*}
            \max&\crk{\ln\prt{n\norm{\Pi\x_0}^2 K}, \ln\prt{\frac{33n\mu^2 K \cL_0}{(2C\sigfn + \sigma^2) L}}, \ln\prt{\frac{n}{K}}, \ln\prt{\frac{n\sumn\norm{\nabla f_i(\bar{x}_0)}^2}{(1-\lambda)^{3/2}{K}}},\right.\\
            &\left.\quad  \ln\prt{\frac{\norm{\Pi\x_0}^2}{\sqrt{1-\lambda}}}}.
        \end{align*}

    \end{corollary}

    \begin{proof}
        We hide the constants that are independent of $n$ and $(1-\lambda)$ in the following. Based on \eqref{eq:fK_bound_order}, we have 
        \begin{equation}
            \label{eq:fK1}
            \begin{aligned}
                &\frac{1}{n}\sumn\E\brk{f(x_{i,K}) - f^*} = \order{\frac{1}{nK}} \tilde{\mathcal{O}}\left\{1 + \frac{\sqrt{n}}{(1-\trw)K} \right.\\
                &\left.\brk{nK\norm{\Pi\x_0}^2 + \frac{n}{K} + \frac{n\sumn\norm{\nabla f_i(\bar{x}_0)}^2}{(1-\trw)^2K}}\exp\prt{-\frac{(1-\trw) K}{\sqrt{n}}}\right.\\
                &\left. + \brk{nK + \frac{ \norm{\Pi\x_0}^2}{(1-\trw)} + \frac{1}{K^2} + \frac{n\sumn\norm{\nabla f_i(\bar{x}_0)}^2}{(1-\trw)^3 K^2}}\exp\prt{-\frac{(1-\trw)K}{\sqrt{n}}}\right\}.
            \end{aligned}
        \end{equation}
        Invoking the definition of $\KTP$ and the relation $(1-\trw)\sim\orderi{\sqrt{1-\lambda}}$ in \eqref{eq:trw} yields the desired result \eqref{eq:KTP_full}.
    \end{proof}

    \section{Simulations}
    \label{sec:sim}

    This section presents two numerical examples that illustrate the performance of DSMT compared with the existing methods over a ring graph (Figure \ref{fig:ring}). We evaluate these algorithms on a strongly convex problem \eqref{eq:logistic} (satisfying the PL condition) and a general nonconvex problem \eqref{eq:ncvx_logistic} to verify the theoretical findings. Overall, the results demonstrate that the incorporation of either the momentum tracking technique or Loopless Chebyshev Acceleration (LCA) enhances the performance of distributed stochastic gradient methods. Furthermore, the DSMT method outperforms existing methods due to the effective combination of these two techniques.

    \begin{figure}
        \centering
        \includegraphics[width=0.5\textwidth]{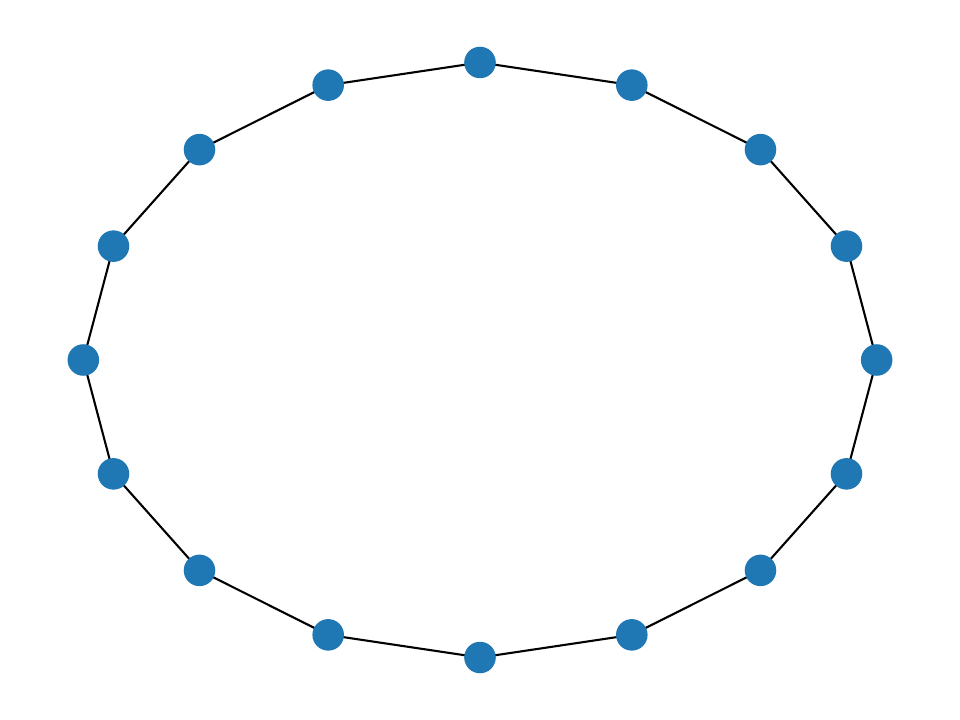}
        \caption{Illustration of ring graph topology with $n = 16$.}
        \label{fig:ring}
    \end{figure}

    \subsection{Logistic Regression}
    \label{subsec:sims_scvx}
    We consider a binary classification problem using logistic regression \eqref{eq:logistic} on the CIFAR-10 \cite{krizhevsky2009learning} dataset. 
    Each agent possesses a distinct local dataset $\mathcal{S}_i{= \crki{(u_j, v_j)}}$ selected from the whole dataset $\mathcal{S}$. {Here, the variable $u_j\in\R^p$ denotes the image input and $v_j\in\R$ represents the label.} In particular, we consider a \textit{heterogeneous} data setting, where data samples are sorted based on their labels and partitioned among the agents. The classifier can then be obtained by solving the following optimization problem using all the agents' local datasets $\mathcal{S}_i, i=1,2,...,n$:

    \begin{equation}
        \label{eq:logistic}
        \begin{aligned}
            &\min_{x\in\R^{p}} f(x) = \frac{1}{n}\sum_{i=1}^n f_i(x),\;f_i(x) := \frac{1}{|\mathcal{S}_i|} \sum_{j\in\mathcal{S}_i} \log\left[1 + \exp(-x^{\T}u_jv_j)\right] + \frac{\rho}{2}\norm{x}^2,
        \end{aligned}
    \end{equation}
    where $\rho$ is set as $0.2$.

    We evaluate the performance of several methods, including DSMT, DSGT \cite{pu2021sharp}, EDAS \cite{huang2021improving}, DSGD \cite{pu2021sharp}, DSGT-HB \cite{gao2023distributed}, centralized SGD (CSGD), and centralized SGDM (CSGDM), in the context of classifying airplanes and trucks on the CIFAR-10 dataset using a constant step size. The results are presented in Figure \ref{fig:logistic}. Additionally, to facilitate comparison, we implement the DSMT method without the LCA technique, denoted as DSMT\_noLCA in subsequent discussions.
    
    \begin{figure}[htbp]
        \centering
        \subfloat[Ring graph, $n=100$, $1-\lambda = 6.6\times10^{-4}$.]{\includegraphics[width=0.49\textwidth]{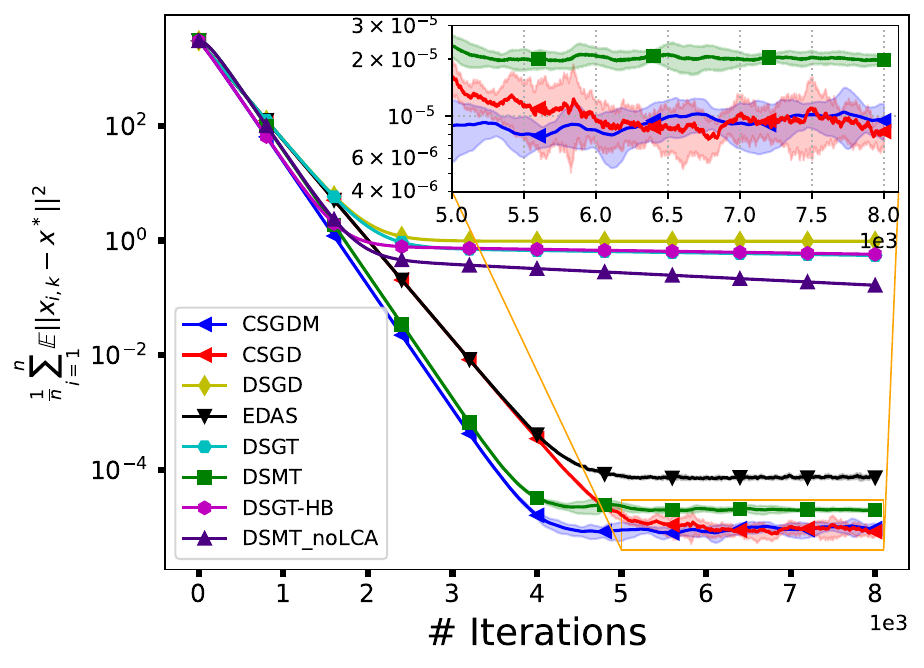}\label{fig:logistic_n100}}
        \subfloat[Ring graph, $n=50$, $1-\lambda = 2.6\times 10^{-3}$.]{\includegraphics[width=0.49\textwidth]{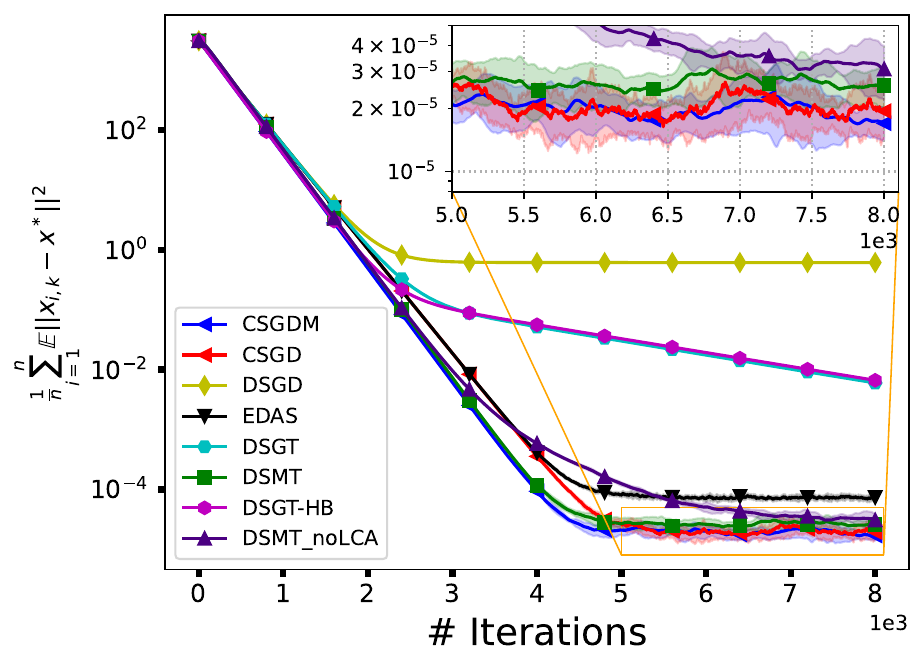}\label{fig:logistic_n50}}
        \caption{Comparison among DSMT, DSGT, EDAS, DSGD, CSGD, and CSGDM for solving Problem \eqref{eq:logistic} on the CIFAR-10 dataset using a constant stepsize. The stepsize is set as $\alpha = 0.01$ for all the methods. The momentum parameter is set as $\beta =\trw$ for DSMT, DSMT\_noLCA, and SGDM. The results are averaged over 10 repeated experiments. The shaded region represents the standard deviation.}
        \label{fig:logistic}
    \end{figure}

    The performance of decentralized methods shown in Figures \ref{fig:logistic_n100} and \ref{fig:logistic_n50} reveals that, as the network size decreases (resulting in a larger value of $(1-\lambda)$), the methods tend to exhibit more comparable performance to centralized SGD. 
    Among these decentralized methods, DSMT achieves the most similar accuracy compared to centralized methods, especially when the network topology is not well-connected (Figure \ref{fig:logistic_n100}). 

    Comparing those decentralized methods incorporating momentum acceleration, they all behave more similar to CSGDM than CSGD, aligning with the arguments in Subsection \ref{subsec:motivation}. However,
    a comparison between DSMT\_noLCA and DSGT-HB reveals the importance of appropriately employing the momentum technique to enhance the practical performance. Straightforward combination may not yield significant improvements, as seen from the performance comparison of DSGT and DSGT-HB in Figure~\ref{fig:logistic}. Such a comparison implies the effectiveness of the momentum tracking technique.

    Comparing the performance of DSMT and DSMT\_noLCA in Figure \ref{fig:logistic}, it can be concluded that the LCA technique indeed accelerates the convergence process. Such an effect becomes more pronounced when the network topology degrades, as evidenced in the transition from Figure \ref{fig:logistic_n50} to Figure \ref{fig:logistic_n100}. 

    In summary, the incorporation of either the momentum tracking technique or LCA method proves effective in enhancing the practical performance of distributed stochastic gradient methods. Notably, the DSMT method outperforms the aforementioned alternatives by combining these two techniques. Such a comparison corroborates the theoretical findings in Section \ref{sec:main}.

    \subsection{Nonconvex Logistic Regression}
    \label{subsec:sims_ncvx}

    In this part, we consider a binary classification problem \eqref{eq:ncvx_logistic} classifying airplanes and trucks on the CIFAR-10 dataset. The parameter $\omega$ is set as $\omega = 0.05$ and $[x]_q$ denotes the $q-$element of $x\in\R^p$. The other settings follow the same as those in Subsection \ref{subsec:sims_scvx}. We compared the aforementioned methods over ring graphs with $n=100$ (Figure \ref{fig:ncvx_logistic_n100}) and $n=50$ (Figure \ref{fig:ncvx_logistic_n50}), respectively.
    \begin{equation}
		\label{eq:ncvx_logistic}
		\begin{aligned}
			& \min_{x\in\R^{p}} f(x) = \frac{1}{n}\sum_{i=1}^n f_i(x),\;f_i(x) := \frac{1}{|\mathcal{S}_i|} \sum_{j\in\mathcal{S}_i} \log\left[1 + \exp(-x^{\T}u_jv_j)\right] + \frac{\omega}{2}\sum_{q=1}^p \frac{[x]_q^2}{1 + [x]_q^2}.
		\end{aligned}
	\end{equation}

    It can be seen from Figure \ref{fig:ncvx_logistic} that solely applying momentum technique to DSGT does not enhance its performance. In Figure \ref{fig:ncvx_logistic_n50}, DSMT\_noLCA demonstrates a slight improvement over DSGT-HB. However, DSMT consistently outperforms other methods, capitalizing on the combined efficacy of momentum tracking technique and LCA method. Notably, the difference between DSMT and EDAS becomes negligible as the number of iteration increases, suggesting that the application of LCA to EDAS could potentially enhance its convergence. Such a phenomenon inspires further research of interest.

    \begin{figure}[htbp]
        \centering
        \subfloat[Ring graph, $n=100$, $1-\lambda = 6.6\times10^{-4}$.]{\includegraphics[width=0.49\textwidth]{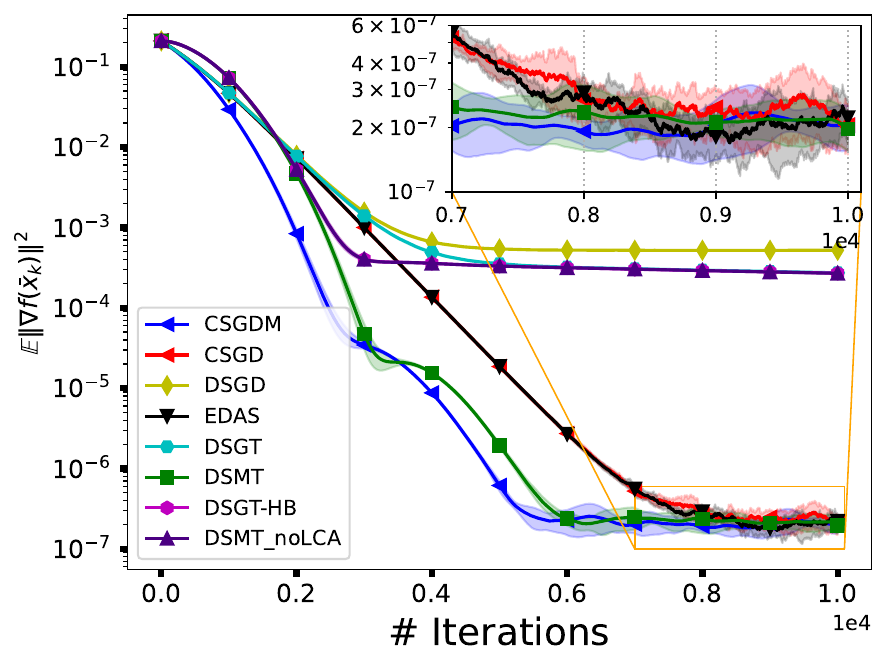}\label{fig:ncvx_logistic_n100}}
        \subfloat[Ring graph, $n=50$, $1-\lambda = 2.6\times 10^{-3}$.]{\includegraphics[width=0.49\textwidth]{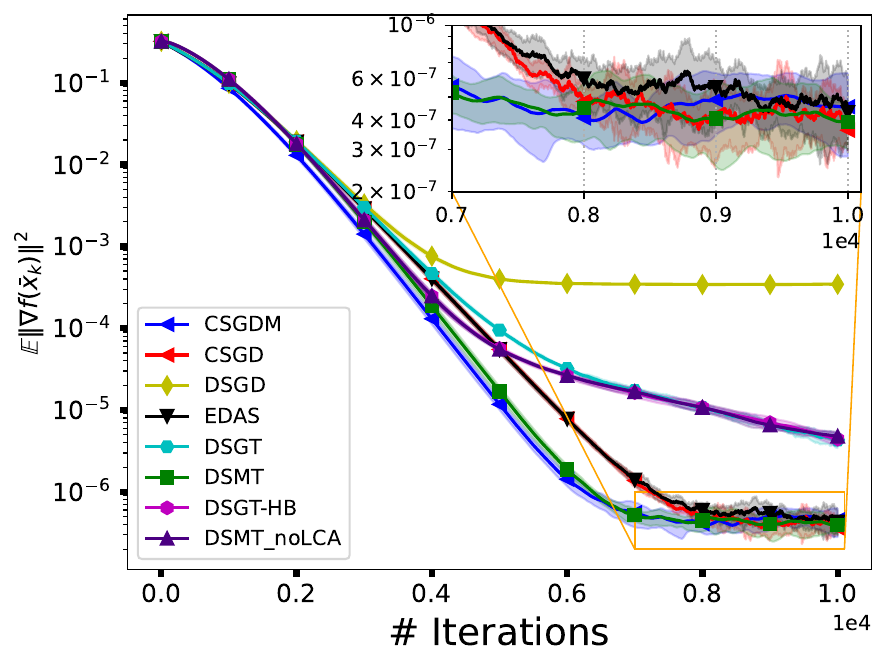}\label{fig:ncvx_logistic_n50}}
        \caption{Comparison among DSMT, DSGT, EDAS, DSGD, CSGD, and CSGDM for solving Problem \eqref{eq:ncvx_logistic} on the CIFAR-10 dataset using a constant stepsize. The stepsizes are set as $0.02$ for all methods. The momentum parameter is set as $\beta = 1 - (1-\trw)/n^{1/3}$ for DSMT, DSMT\_noLCA, and SGDM. The results are averaged over 10 repeated experiments. The shaded region represents the standard deviation.}
        \label{fig:ncvx_logistic}
    \end{figure}

    \section{Conclusion}
    \label{sec:con}

    This paper focuses on addressing the distributed stochastic optimization problem over networked agents. The proposed algorithm, Distributed Stochastic Momentum Tracking (DSMT), leverages momentum tracking technique as well as the Loopless Chebyshev Acceleration (LCA) method to enhance the performance of distributed stochastic gradient methods over networks. In particular, DSMT shortens the transient times for decentralized stochastic gradient methods without requiring multiple communication per iteration for smooth objective functions with or without the Polyak-{\L}ojasiewicz (PL) condition under the most general variance condition in distributed settings. Such a condition for stochastic gradients enables the wide range of application of DSMT. Experimental results also corroborate to such theoretical findings. The momentum tracking technique is also of independent interest and can potentially inspire future development of distributed optimization algorithms.

\section{Statements and Declarations}
The authors declare that there is no conflict of interest. The CIFAR-10 dataset is online at \href{https://www.cs.toronto.edu/\~kriz/cifar.html}{https://www.cs.toronto.edu/$\sim$kriz/cifar.html}.

\backmatter

\begin{appendices}

\section{Proofs for the General Nonconvex Case} 

    \subsection{Proof of Lemma \ref{lem:avg}}
    \label{app:avg}

    Define $\ub{x}_k^{\T} := \frac{1}{n}\1^{\T}[\tx_k]_{n+1:2n}$ and $\wb{x}_k^{\T} := \frac{1}{2n}\1^{\T}\tx_k$. From \eqref{eq:tx}, we have 
    \begin{equation}
        \label{eq:xkp1}
        \begin{aligned}
            \wb{x}_{k + 1} &= \frac{1}{2}\brk{(2 + \eta_w)\bar{x}_k- \eta_w\ub{x}_k} - \alpha \bar{y}_k,\\
            \x_{k + 1} &= (1 + \eta_w)W\prt{\x_k - \alpha \y_k} - \eta_w \prt{[\tx_k]_{n+1:2n} - \alpha \y_k},\\
            [\tx_{k + 1}]_{n+1:2n} &= \x_k - \alpha \y_k.
        \end{aligned}
    \end{equation}

    Therefore, it suffices to show $\bar{x}_k = \ub{x}_k = \wb{x}_k$ for any $k\geq 0$, which we prove by induction. When $k=0$, $\bar{x}_0 = \ub{x}_0 = \wb{x}_0$ since $\tx_0 = \xsh{0}$. Suppose $\bar{x}_k = \ub{x}_k = \wb{x}_k$ for some $k\ge 0$, then \eqref{eq:xkp1} leads to 
    \begin{align*}
        \wb{x}_{k + 1} &= \bar{x}_k - \alpha \bar{y}_k,\\
        \bar{x}_{k + 1} &= (1 + \eta_w)\bar{x}_k - \eta_w\ub{x}_k -\alpha \bar{y}_k = \bar{x}_k - \alpha\bar{y}_k,\\
        \ub{x}_{k + 1} &= \bar{x}_k - \alpha \bar{y}_k.
    \end{align*}

    Hence, we have $\bar{x}_k = \ub{x}_k = \wb{x}_k$ for any $k\geq 0$. Similar {line of analysis} can show that $\wb{y}_{k + 1} = \bar{y}_{k+1} = \ub{y}_{k + 1} = \bar{y}_k + \bar{z}_{k + 1} - \bar{z}_k$. The update form for $\bar{y}_k$ and the initialization $\y_0 = \z_0$ imply $
    \bar{y}_k = \bar{z}_k$ for any $k\geq 0$.

    \subsection{Proof of Lemma \ref{lem:dk}}
    \label{app:dk}
    Recall from Lemma \ref{lem:avg} that $\bar{x}_{k + 1} = \bar{x}_k - \alpha\bar{z}_k$.
    We first show $\bar{d}_{k + 1} = \bar{d}_k - \alpha \bar{g}_k$. 
    When $k = 0$, we have 
    \begin{align*}
        \bar{d}_1 - \bar{d}_0 &= \frac{1}{1-\beta}\bar{x}_1 - \frac{\beta}{1-\beta}\bar{x}_0 - \bar{x}_0= \frac{1}{1-\beta}\prt{\bar{x}_0 - \alpha \bar{z}_0} - \frac{1}{1-\beta}\bar{x}_0= -\alpha\bar{g}_0,
    \end{align*}
    where the last equality holds by initializing $\z_0 = (1-\beta)\g_0$.

    For the case $k\geq 1$, we have 
    \begin{align*}
        \bar{d}_{k + 1} - \bar{d}_k = \frac{1}{1-\beta}\prt{\bar{x}_{k + 1} -\beta\bar{x}_k - \bar{x}_k + \beta\bar{x}_{k-1}}= \frac{1}{1-\beta}\prt{-\alpha \bar{z}_k + \beta\alpha\bar{z}_{k-1}}= -\alpha\bar{g}_k.
    \end{align*}

    Next we show $\bar{d}_{k} - \bar{x}_k = -\frac{\alpha\beta}{1-\beta}\bar{z}_{k-1}$, which is because for $k\geq 1$,
    \begin{align}
        \label{eq:bdk_bxk}
        \bar{d}_k - \bar{x}_k = \frac{1}{1-\beta}\bar{x}_k - \frac{\beta}{1-\beta}\bar{x}_{k-1} - \bar{x}_k
        = \frac{\beta}{1-\beta}\prt{\bar{x}_k - \bar{x}_{k-1}}= -\frac{\alpha \beta}{1-\beta}\bar{z}_{k-1}.
    \end{align}
    Due to $\z_{-1} = \0$ and $\bar{d}_0 = \bar{x}_0$, the relation \eqref{eq:bdk_bxk} also holds for $k=0$.

    \subsection{Proof of Lemma \ref{lem:fx_fd}}
    \label{app:fx_fd}

    Given that $\bar{x}_k = \bar{d}_k + \alpha\beta\bar{z}_{k-1}/(1-\beta)$, $k\geq 0$ from Lemma \ref{lem:dk}, we have 
    \begin{align}
        &f(\bar{x}_k) - f^* \leq f(\bar{d}_k) - f^* + \frac{\alpha\beta}{1-\beta}\inpro{\nabla f(\bar{d}_k) - \nabla f(\bar{x}_k) + \nabla f(\bar{x}_k), \bar{z}_{k-1}}\nonumber\\
        &\quad + \frac{\alpha^2\beta^2 L}{2(1-\beta)^2}\norm{\bar{z}_{k-1}}^2\label{eq:fx_fd}\\
        &\leq f(\bar{d}_k) - f^* + \frac{\alpha\beta}{2(1-\beta)}\norm{\nabla f(\bar{x}_k)}^2 + \frac{\alpha\beta[1 + 3\alpha\beta L/(1-\beta)] }{2(1-\beta)}\norm{\bar{z}_{k-1}}^2,
    \end{align}
    where we invoke $\normi{\nabla f(\bar{x}_k) - \nabla f(\bar{d}_k)}\leq \alpha\beta L\normi{\bar{z}_{k - 1}}/(1-\beta)$ for $k\geq 0$.
    Taking the full expectation on both sides yields the desired result.

    \subsection{Proof of Lemma \ref{lem:descent_dk}}
    \label{app:descent_dk}

    In the light of Lemma~\ref{lem:dk} that $\bar{d}_{k + 1} = \bar{d}_k - \alpha\bar{g}_k$ for $k\geq 0$, we have from the descent lemma and Assumption \ref{as:abc} that 
    \begin{equation}
        \label{eq:de}
        \begin{aligned}
            &\condE{f(\bar{d}_{k + 1})}{\cF_k} \leq f(\bar{d}_k) - \alpha\inpro{\nabla f(\bar{d}_k) - \nabla f(\bar{x}_k), \frac{1}{n}\sumn\nabla f_i(x_{i,k})} \\
            &\quad - \alpha\inpro{\nabla f(\bar{x}_k),\frac{1}{n}\sumn\nabla f_i(x_{i,k}) } + \frac{L}{2}\condE{\norm{\alpha \bar{g}_k}^2}{\cF_k}\\
            &\leq f(\bar{d}_k) + \alpha \norm{\nabla f(\bar{d}_k) - \nabla f(\bar{x}_k)}^2 + \frac{\alpha}{4}\norm{\frac{1}{n}\sumn\nabla f_i(x_{i,k})}^2 - \frac{\alpha}{2}\norm{\nabla f(\bar{x}_k)}^2 \\
            &\quad - \frac{\alpha(1-\alpha L)}{2}\norm{\frac{1}{n}\sumn\nabla f_i(x_{i,k})}^2 + \frac{\alpha}{2}\norm{\nabla f(\bar{x}_k) - \frac{1}{n}\sumn\nabla f_i(x_{i,k})}^2\\
            &\quad + \frac{\alpha^2 L}{2} \condE{\norm{\bar{g}_k - \frac{1}{n}\sumn\nabla f_i(x_{i,k})}^2}{\cF_k}\\
            &\leq f(\bar{d}_k) + \alpha L^2\norm{\bar{d}_k - \bar{x}_k}^2 - \frac{\alpha }{2}\norm{\nabla f(\bar{x}_k)}^2 - \frac{\alpha}{8}\norm{\frac{1}{n}\sumn\nabla f_i(x_{i,k})}^2 \\
            &\quad + \frac{\alpha L^2}{2n}\sumn\norm{\bar{x}_k - x_{i,k}}^2+ \frac{\alpha^2 L}{n}\brk{\frac{C}{n}\sumn \prt{f_i(x_{i,k}) - f_i^*} + \sigma^2},
        \end{aligned}
    \end{equation}
    where we let $\alpha\leq 1/(4L)$ and invoke Assumption \ref{as:abc} for the last inequality. Noting that 
    \begin{align}
        f_i(x_{i,k}) - f_i^* &\leq f_i(\bar{x}_k) - f_i^* + \inpro{\nabla f_i(\bar{x}_k), x_{i,k} - \bar{x}_k} + \frac{L}{2}\norm{x_{i,k} - \bar{x}_k}^2\nonumber\\
        &\leq f_i(\bar{x}_k) - f_i^* + \frac{1}{2L}\norm{\nabla f_i(\bar{x}_k)}^2 + L\norm{x_{i,k} - \bar{x}_k}^2\nonumber\\
        &\leq 2\prt{f_i(\bar{x}_k) - f_i^*} + L\norm{x_{i,k} - \bar{x}_k}^2,\label{eq:fi_f}
    \end{align}
    we have 
    \begin{equation}
        \label{eq:fxik}
        \begin{aligned}
            \frac{1}{n}\sumn \prt{f_i(x_{i,k}) - f_i^*} &\leq
            2\prt{f(\bar{x}_k) - f^*} + 2\sigfn + \frac{L}{n}\norm{\x_k - \1\bar{x}_k^{\T}}^2.
        \end{aligned}
    \end{equation}
    By the relation $\bar{d}_k - \bar{x}_k = -\alpha\beta\bar{z}_{k-1}/(1-\beta)$ and $\z_{-1} = \0$, we have 
    \begin{equation}
        \label{eq:db_xb}
        \norm{\bar{d}_k - \bar{x}_k}^2 = \frac{\alpha^2\beta^2}{(1-\beta)^2}\norm{\bar{z}_{k-1}}^2, k\geq 0.
    \end{equation}
Combining \eqref{eq:de}-\eqref{eq:db_xb}, taking the full expectation, letting $\alpha\leq 1/(2C)$, and invoking Lemma \ref{lem:fx_fd} yield the desired result for $k\geq 0$. 


        
    \subsection{Proof of Lemma \ref{lem:zk}}
    \label{app:zk}
    According to Lemma \ref{lem:avg}, we have 
    \begin{equation}
        \label{eq:zk_avg}
        \begin{aligned}
            \bar{z}_k &= \beta\bar{z}_{k-1} + (1-\beta)\prt{\bar{g}_k - \frac{1}{n}\sumn\nabla f_i(x_{i,k})} + (1-\beta)\brk{\frac{1}{n}\sumn\nabla f_i(x_{i,k}) - \nabla f(\bar{x}_k)}\\
            &\quad +(1-\beta)\nabla f(\bar{x}) .
        \end{aligned}
    \end{equation}
Then, by Assumption \ref{as:abc} and Jensen's inequality, we have
    \begin{equation}
        \label{eq:zk_s1}
        \begin{aligned}
            &\condE{\norm{\bar{z}_k}^2}{\cF_k} \leq \beta\norm{\bar{z}_{k-1}}^2 + \frac{(1-\beta)^2}{n}\brk{\frac{C}{n}\sumn\prt{f_i(x_{i,k}) - f_i^*} + \sigma^2} \\
            &\quad+ 2(1-\beta)\norm{\nabla f(\bar{x}_k)}^2 + \frac{2(1-\beta)L^2}{n}\norm{\x_k - \1\bar{x}_k^{\T}}^2\\
            &\leq \beta\norm{\bar{z}_{k-1}}^2 + \frac{2C(1-\beta)^2 }{n}\brk{f(\bar{x}_k)-f^*} + 2(1-\beta)\norm{\nabla f(\bar{x}_k)}^2\\
            &\quad + \frac{(1-\beta)^2(2C\sigfn + \sigma^2)}{n}+ \frac{2(1-\beta)L}{n}\prt{L + \frac{C(1-\beta)}{n}}\norm{\x_k-\1\bar{x}_k^{\T}}^2.
        \end{aligned}
    \end{equation}
    Taking the full expectation on both sides and invoking Lemma \ref{lem:fx_fd} yield the desired result.

    \subsection{Proof of Lemma \ref{lem:zk_nf}}
    \label{app:zk_cons}
    First consider 
    \begin{align}
        &\z_k - \nabla F(\Barx{k}) = \beta\z_{k-1} + (1-\beta)\g_k - \nabla F(\Barx{k})\nonumber\\
        &= \beta\brk{\z_{k - 1} - \nabla F(\Barx{k-1})} + (1-\beta)\brk{\g_k - \nabla F(\x_k)} + (1-\beta)\brk{\nabla F(\x_k) - \nabla F(\Barx{k})}\nonumber\\
        &\quad - \beta\brk{\nabla F(\Barx{k}) - \nabla F(\Barx{k - 1})},\;k\geq 1.  \label{eq:zknf}
    \end{align}
    From \eqref{eq:zknf} and Jensen's inequality, we have 
    \begin{equation}
        \label{eq:zknf_norm1}
        \begin{aligned}
            &\condE{\norm{\z_k - \nabla F(\Barx{k})}^2}{\cF_k} \leq \beta\norm{\z_{k - 1} - \nabla F(\Barx{k-1})}^2\\
            &\quad + 2(1-\beta)\norm{\nabla F(\x_k) - \nabla F(\Barx{k})}^2+ \frac{2\beta^2}{1-\beta}\norm{\nabla F(\Barx{k}) - \nabla F(\Barx{k - 1})}^2 \\
            &\quad +(1-\beta)^2\condE{\norm{\g_k - \nabla F(\x_k)}^2}{\cF_k}\\
            &\leq \beta\norm{\z_{k - 1} - \nabla F(\Barx{k-1})}^2 + \frac{2\alpha^2\beta^2 n L^2}{1-\beta}\norm{\bar{z}_{k-1}}^2 + 2(1-\beta)L^2\norm{\Pi\x_k}^2\\
            &\quad + n(1-\beta)^2(2C\sigfn + \sigma^2)+ n(1-\beta)^2C\crk{2\brk{f(\bar{x}_k) - f^*} + \frac{L}{n}\norm{\Pi\x_k}^2},
        \end{aligned}
    \end{equation}
    where Assumption \ref{as:abc} and the upper bound \eqref{eq:nF_diff} are invoked for the last inequality. The inequality \eqref{eq:nF_diff} holds due to Lemma \ref{lem:avg} that $\bar{x}_{k } = \bar{x}_{k-1} - \alpha\bar{z}_{k-1}$
    \begin{equation}
        \label{eq:nF_diff}
        \begin{aligned}
            &\E\brk{\norm{\nabla F(\Barx{k}) - \nabla F(\Barx{k-1})}^2} \leq nL^2\E\brk{\norm{\bar{x}_{k} - \bar{x}_{k-1}}^2}\leq \alpha^2n L^2\E\brk{\norm{\bar{z}_{k-1}}^2}.
        \end{aligned}
    \end{equation}
    Taking full expectation on both sides of \eqref{eq:zknf_norm1}, and substituting the result of Lemmas \ref{lem:fx_fd} yield that
    \begin{align*}
       &\E\brk{\norm{\z_k - \nabla F(\Barx{k})}^2}\leq \beta\E\brk{\norm{\z_{k - 1} - \nabla F(\Barx{k-1})}^2} + \frac{2\alpha^2\beta^2 n L^2}{1-\beta}\E\brk{\norm{\bar{z}_{k-1}}^2}\\
       &\quad + 2(1-\beta)L^2\E\brk{\norm{\Pi\x_k}^2}+ n(1-\beta)^2(2C\sigfn + \sigma^2)+ n(1-\beta)^2C\crk{2\E\brk{f(\bar{d}_k) - f^*}\right.\\
       &\left.\quad + \frac{2\alpha\beta}{1-\beta}\E\brk{\norm{\nabla f(\bar{x}_k)}^2} + \frac{2\alpha\beta}{1-\beta}\E\brk{\norm{\bar{z}_{k-1}}^2} + \frac{L}{n}\E\brk{\norm{\Pi\x_k}^2}}\\
       &\leq \beta\E\brk{\norm{\z_{k - 1} - \nabla F(\Barx{k-1})}^2} + 2\alpha n\prt{\frac{\alpha L^2}{1-\beta} + (1-\beta)C}\E\brk{\norm{\bar{z}_{k - 1}}^2}\\
       &\quad + 2(1-\beta)L(L+C)\E\brk{\norm{\Pi\x_k}^2}+ n(1-\beta)^2(2C\sigfn + \sigma^2)\\
       &\quad + 2n(1-\beta)^2C\E\brk{f(\bar{d}_k) - f^*} + 2\alpha n(1-\beta)C\E\brk{\norm{\nabla f(\bar{x}_k)}^2}, \;k\geq 1.
    \end{align*}
    Letting $\alpha \leq (1-\beta)/L$ leads to the desired result for $k\geq 1$. For the case $k=0$, recall that $\x_{-1} = \x_0$, $\z_{-1} = \0$ and $\z_{0}= (1-\beta)\g_0$, then 
    \begin{equation}
        \label{eq:zknf_k9}
        \begin{aligned}
            &\z_0 - \nabla F(\Barx{0}) = \beta\brk{\z_{ - 1} - \nabla F(\Barx{-1})} + (1-\beta)\brk{\g_0 - \nabla F(\x_0)}\\
            &\quad + (1-\beta)\brk{\nabla F(\x_0) - \nabla F(\Barx{0})}- \beta\brk{\nabla F(\Barx{0}) - \nabla F(\Barx{- 1})}.
        \end{aligned}
    \end{equation}
    Since $\brki{\nabla F(\Barx{0}) - \nabla F(\Barx{- 1})} = \0$ due to $\x_{-1} = \x_0$, similar to the derivations of \eqref{eq:zknf_norm1}, we have
    \begin{equation}
        \label{eq:zknf_k0}
        \begin{aligned}
            &\E\brk{\norm{\z_0 - \nabla F(\Barx{0})}^2} \leq \beta\E\brk{\norm{\z_{-1} - \nabla F(\Barx{-1})}^2} + (1-\beta)L^2\norm{\Pi\x_0}^2\\
            &\quad + n(1-\beta)^2(2C\sigfn + \sigma^2)+ n(1-\beta)^2C\crk{2\E\brk{f(\bar{x}_0) - f^*} + \frac{L}{n}\norm{\Pi\x_0}^2}\\
            &\leq \beta\E\brk{\norm{\z_{-1} - \nabla F(\Barx{-1})}^2} + (1-\beta)L(L+C)\norm{\Pi\x_{0}}^2\\
            &\quad + n(1-\beta)^2(2C\sigfn + \sigma^2)+ 2n(1-\beta)^2C\E\brk{f(\bar{d}_{0}) - f^*}, 
        \end{aligned}
    \end{equation}
    where we invoke that $\bar{x}_{-1} = \bar{x}_0 = \bar{d}_0$. As a consequence, combining the recursions for $k=0$ and $k\geq 1$ yields the desired result.

    \subsection{Proof of Lemma \ref{lem:cons_R}}
    \label{app:cons_R}

    Note that $\tpi\tW = \tpi\tW\tpi$. Then, 
    \begin{equation}
        \label{eq:tpi_x}
        \begin{aligned}
            \condE{\norm{\tpx_{k + 1}}^2}{\cF_k}&= \condE{\norm{\tpi\tW\tpx_k - \alpha \tpi\tW\tpi\ysh{k}}^2}{\cF_k}\\
            &\leq \frac{1}{q}\norm{\tpi\tW\tpx_k}^2 + \frac{\alpha^2c_0\trw^2}{1-q}\condE{\norm{\Pi\y_k}^2}{\cF_k},
        \end{aligned}
    \end{equation}
    where we invoke Jensen's inequality for some $q>0$ (to be determined later) and Lemma \ref{lem:lca} for the last inequality.

    For the term $\normi{\tpi\tW\tpx_k}^2$, we have 
    \begin{equation}
        \label{eq:tpiw_x}
        \begin{aligned}
            \frac{1}{q}\condE{\norm{\tpi\tW\tpx_k}^2}{\cF_{k-1}} &\leq \frac{1}{q^2}\norm{\tpi\tW^2\tpx_{k-1}}^2 + \frac{\alpha^2c_0\trw^4}{q(1-q)}\condE{\norm{\Pi\y_{k-1}}^2}{\cF_{k-1}}.
        \end{aligned}
    \end{equation}
Choosing $q = \trw$ in \eqref{eq:tpi_x} and \eqref{eq:tpiw_x} and repeating such procedures lead to
    \begin{equation}
        \label{eq:tpi_x_sum}
        \begin{aligned}
            \E\brk{\norm{\tpx_k}^2} &\leq \frac{1}{\trw^k}\norm{\tpi\tW^k\tpx_0}^2 + \frac{\alpha^2 c_0\trw}{1-\trw}\sum_{t=0}^{k-1}\trw^{k-t}\E\brk{\norm{\Pi\y_t}^2}\\
            &\leq c_0\trw^k\norm{\Pi\x_0}^2 + \frac{\alpha^2 c_0\trw}{1-\trw}\sum_{t=0}^{k-1}\trw^{k-t}\E\brk{\norm{\Pi\y_t}^2}:= \cR_k^x,
        \end{aligned}
    \end{equation}
    where we apply Lemma \ref{lem:lca} for the last inequality. 
    For the term $\cR_k^x$, we have 
    \begin{align}
        \cR_{k + 1}^x &= \trw\cR_k^x + \frac{\alpha^2 c_0\trw^2}{1-\trw}\E\brk{\norm{\Pi\y_k}^2}\leq \trw\cR_k^x + \frac{\alpha^2 c_0\trw}{1-\trw}\E\brk{\norm{\tpy_k}^2}\label{eq:Rx}.
    \end{align}
Similarly, we construct a sequence $\crki{\cR_k^y}$ that bounds the term $\E\brki{\normi{\tpi\ty}^2}$ { in the remaining proofs. We begin by handling the term $\z_{k + 1} - \z_k$ below.} 
Relation \eqref{eq:zknf} guides us to rewrite $(\z_{k + 1} - \z_k)$ as in \eqref{eq:zk_diff}, so that we can take advantage of the additional coefficient $(1-\beta)$ and reduce the impact of the data heterogeneity due to $\sumn\normi{\nabla f_i(x_{i,k})}^2$ as mentioned in Subsection~\ref{subsec:motivation}. 
    \begin{equation}
        \label{eq:zk_diff}
        \begin{aligned}
            &\z_{k + 1} - \z_k =-(1 - \beta)\prt{\z_k - \nabla F(\Barx{k})} + (1 - \beta) \brk{\g_{k + 1} - \nabla F(\x_{k + 1})}\\
            &\quad + (1-\beta)\brk{\nabla F(\x_{k + 1}) - \nabla F(\Barx{k+1})}+(1-\beta)\brk{\nabla F(\Barx{k+1}) - \nabla F(\Barx{k})}.
        \end{aligned}
    \end{equation}
    The above relation follows from \eqref{eq:z}.
    Therefore, we can rewrite the update of $\ty$ as follows.
    \begin{equation}
        \label{eq:ty_de}
        \begin{aligned}
            &\ty_{k + 1} = \tW\ty_k - (1-\beta)\tW\brk{\z_k - \nabla F(\Barx{k})}_{\#} + (1-\beta)\tW\brk{\g_{k + 1} - \nabla F(\x_{k + 1})}_{\#}\\
            &\quad + (1-\beta)\tW\brk{\nabla F(\x_{k + 1}) - \nabla F(\Barx{k+1})}_{\#} + (1-\beta)\tW\brk{\nabla F(\Barx{k + 1}) - \nabla F(\Barx{k})}_{\#}.
        \end{aligned}
    \end{equation}

    Denote $\cA:= \ty_k - (1-\beta)\brki{\z_k - \nabla F(\Barx{k})}_{\#} + (1-\beta)\brki{\nabla F(\x_{k + 1}) - \nabla F(\Barx{k + 1})}_{\#} + (1-\beta)\brki{\nabla F(\Barx{k + 1}) - \nabla F(\Barx{k})}_{\#}$. We then have
    \begin{equation}
        \label{eq:ty_s1}
        \begin{aligned}
            & \condE{\norm{\tpy_{k + 1}}^2}{\cF_{k + 1}} = \norm{\tpi\tW\tpi\cA}^2\\
            &\quad + (1-\beta)^2\condE{\norm{\tpi\tW\tpi\brk{\g_{k + 1} - \nabla F(\x_{k + 1})}_{\#}}^2}{\cF_{k + 1}}\\
            &\quad + 2\condE{\inpro{\tpi\tW\tpi\cA, (1-\beta)\tpi\tW\tpi\brk{\g_{k + 1} - \nabla F(\x_{k + 1})}_{\#}}}{\cF_{k + 1}}\\
            &\leq \frac{1}{\trw}\norm{\tpi\tW\tpy_k}^2 + nc_0\trw^2(1-\beta)^2\prt{2C\sigfn + \sigma^2}\\
            &\quad + \frac{3(1-\beta)^2c_0\trw^2}{1-\trw}\brk{\norm{\z_k - \nabla F(\Barx{k})}^2 + L^2\norm{\Pi\x_{k + 1}}^2 + \alpha^2nL^2 \norm{\bar{z}_k}^2}\\
            &\quad + nc_0\trw^2C(1-\beta)^2\crk{2\brk{f(\bar{x}_{k + 1}) - f^*} + \frac{L}{n}\norm{\Pi\x_{k + 1}}^2},
        \end{aligned}
    \end{equation}
    where we invoke Lemma \ref{lem:lca}, Assumptions \ref{as:smooth} and \ref{as:abc}, and \eqref{eq:fxik} for the last inequality. Similarly, we can obtain the recursion for $\condEi{\normi{\tpi\tW\tpy_k}^2}{\cF_k}$: 
    \begin{equation}
        \label{eq:twy_s1}
        \begin{aligned}
            &\frac{1}{\trw}\condE{\norm{\tpi\tW\tpy_k}^2}{\cF_k} \leq \frac{1}{\trw^2}\norm{\tpi\tW^2\tpy_{k-1}}^2 + nc_0\trw^3(1-\beta)^2(2C\sigfn + \sigma^2) \\
            &\quad + \frac{3(1-\beta)^2c_0\trw^3}{1-\trw}\brk{\norm{\z_{k-1} - \nabla F(\Barx{k-1})}^2 + L^2\norm{\Pi\x_k}^2 + \alpha^2nL^2\norm{\bar{z}_{k-1}}^2}\\
            &\quad + nc_0\trw^3(1-\beta)^2C\crk{2\brk{f(\bar{x}_k) - f^*} + \frac{L}{n}\norm{\Pi\x_k}^2}.
        \end{aligned}
    \end{equation}
    Taking the full expectation in \eqref{eq:ty_s1} and \eqref{eq:twy_s1} and {using the similar steps as in the preceding analysis} lead to \pagebreak
    \begin{equation}
        \label{eq:piy}
        \begin{aligned}
            &\E\brk{\norm{\tpi\ty_k}^2} \leq \frac{1}{\trw^{k}}\E\brk{\norm{\tpi\tW^{k}\tpy_0}^2} + nc_0\trw^2(1-\beta)^2(2C\sigfn + \sigma^2)\sum_{t=0}^{k-1}\trw^{k-1-t}\\
            &\quad + \frac{3(1-\beta)^2c_0\trw^2}{1-\trw}\sum_{t=0}^{k-1}\trw^{k-1-t}\crk{\E\brk{\norm{\z_{t} - \nabla F(\Barx{t})}^2} + L^2\E\brk{\norm{\Pi\x_{t+1}}^2}\right.\\
            &\left.\quad + \alpha^2 L^2n\E\brk{\norm{\bar{z}_{t}}^2}}+ nc_0\trw^2(1-\beta)^2C\sum_{t=1}^{k}\trw^{k-t}\crk{2\E\brk{f(\bar{x}_t) - f^*} + \frac{L}{n}\E\brk{\norm{\Pi\x_t}^2}}\\
            &\leq c_0\trw^k \E\brk{\norm{\Pi\y_0}^2}+ \frac{3(1-\beta)^2c_0\trw^2}{1-\trw}\sum_{t=0}^{k-1}\trw^{k-1-t}\crk{\E\brk{\norm{\z_{t} - \nabla F(\Barx{t})}^2}\right.\\
            &\left.\quad + L^2\E\brk{\norm{\Pi\x_{t+1}}^2} + \alpha^2 L^2n\E\brk{\norm{\bar{z}_{t}}^2}} + nc_0\trw^2(1-\beta)^2(2C\sigfn + \sigma^2)\sum_{t=0}^{k-1}\trw^{k-1-t}\\
            &\quad + nc_0\trw^2(1-\beta)^2C\sum_{t=1}^{k}\trw^{k-t}\crk{2\E\brk{f(\bar{x}_t) - f^*} + \frac{L}{n}\E\brk{\norm{\Pi\x_t}^2}}\\
            &\quad :=\cR_k^y,\; k\geq 1,
        \end{aligned}
    \end{equation}
    where we invoke Lemma \ref{lem:lca} for the last inequality.
    Then we can obtain the recursion for $\cR_k^y$ according to \eqref{eq:piy}:
    \begin{equation}
        \label{eq:cRy}
        \begin{aligned}
            \cR_{k + 1}^y &= \trw\cR_k^y + \frac{3(1-\beta)^2c_0\trw^2}{1-\trw}\crk{\E\brk{\norm{\z_{k}-\nabla F(\Barx{k})}^2} + L^2\norm{\Pi\x_{k+1}}^2\right.\\
            &\left.\quad + \alpha^2L^2n\E\brk{\norm{\bar{z}_k}^2}}+ nc_0\trw^2(1-\beta)^2(2C\sigfn + \sigma^2)\\
            &\quad + nc_0\trw^2(1-\beta)^2C\crk{2\E\brk{f(\bar{x}_{k+1}) - f^*} + \frac{L}{n}\E\brk{\norm{\Pi\x_{k+1}}^2}}, \;k \geq 1.
        \end{aligned}
    \end{equation}

    Note that \eqref{eq:piy} defines $\cR_1^y$ as follows 
    \begin{align*}
        \cR_1^y &=c_0\trw \E\brk{\norm{\Pi\y_0}^2}+ \frac{3(1-\beta)^2c_0\trw^2}{1-\trw}\crk{\E\brk{\norm{\z_{0} - \nabla F(\Barx{0})}^2} + L^2\E\brk{\norm{\Pi\x_{1}}^2}\right.\\
        &\left.\quad + \alpha^2 L^2n\E\brk{\norm{\bar{z}_{0}}^2}}+ nc_0\trw^2(1-\beta)^2(2C\sigfn + \sigma^2)\\
        &\quad + nc_0\trw^2(1-\beta)^2C\crk{2\E\brk{f(\bar{x}_1) - f^*} + \frac{L}{n}\E\brk{\norm{\Pi\x_1}^2}}.
    \end{align*}
Hence, we can define $\cR_0^y:= c_0\E\brki{\normi{\Pi\y_0}^2} + nc_0\trw(1-\beta)^2(2C\sigfn + \sigma^2)$ so that the recursion \eqref{eq:cRy} also holds for $k=0$.

    Noting that $\E\brki{\normi{\Pi\x_{k+1}}^2}\leq \E\brki{\normi{\tpx_{k + 1}}^2}\leq \cR_{k+1}^x$ and $\E\brki{\normi{\tpy_k}^2}\leq \cR_k^y$, the desired result follows by invoking \eqref{eq:Rx} and Lemmas \ref{lem:fx_fd}-\ref{lem:zk_nf}:
    \begin{align*}
        &\cR_{k + 1}^y 
        \leq \brk{\trw + \frac{3\alpha^2c_0^2(1-\beta)^2 L(L+C)}{(1-\trw)^2}}\cR_k^y \\
        &\quad + \frac{3c_0(1-\beta)^2L}{1-\trw}\brk{L+C + 2(L+C) + 2\alpha(L+C)^2 + \alpha CL}\cR_k^x\\
        &\quad + \frac{3c_0(1-\beta)^2\beta}{1-\trw}\E\brk{\norm{\z_{k - 1} - \nabla F(\Barx{k-1})}^2} \\
        &\quad + 3\alpha n (1-\beta)c_0(L+C)\brk{\frac{2(1-\beta)}{1-\trw} + 1 + \frac{4\alpha^2 CL}{3(1-\beta)}}\E\brk{\norm{\bar{z}_{k-1}}^2}\\
        &\quad + 2nc_0(1-\beta)^2C\brk{\frac{3(1-\beta)^2}{1-\trw} + \frac{3\alpha(1-\beta)(C+L)}{n} + 2}\E\brk{f(\bar{d}_k) - f^*}\\
        &\quad + 4\alpha n c_0(1-\beta)\brk{C + \frac{6(1-\beta)^2C}{4(1-\trw)} + \frac{9(C+L)(1-\beta)}{4}}\E\brk{\norm{\nabla f(\bar{x}_k)}^2}\\
        &\quad + nc_0(1-\beta)^2\brk{1 + \frac{3(1-\beta)^2}{1-\trw} + \frac{3\alpha(1-\beta)(L+C)}{n} + \frac{2\alpha^2 CL}{n}}(2C\sigfn + \sigma^2)\\
        &\leq \frac{1+\trw}{2}\cR_k^y + \frac{15c_0(1-\beta)^2L(L+C)}{1-\trw}\cR_k^x + \frac{3c_0(1-\beta)^2\beta}{1-\trw}\E\brk{\norm{\z_{k - 1} - \nabla F(\Barx{k-1})}^2}\\
        &\quad + 9\alpha n c_0(1-\beta)(C+L)\E\brk{\norm{\bar{z}_{k - 1}}^2} + 12nc_0(1-\beta)^2C\E\brk{f(\bar{d}_k) - f^*}\\
        &\quad + 20\alpha n c_0(1-\beta)(C+L)\E\brk{\norm{\nabla f(\bar{x}_k)}^2}+ 6nc_0(1-\beta)^2(2C\sigfn + \sigma^2),
    \end{align*}
    where we let $1-\beta\leq 1-\trw$ and the stepsize satisfy
    \begin{align*}
        \alpha\leq\min\crk{\frac{1}{2(L+C)}, \sqrt{\frac{1-\trw}{6c_0^2L(L+C)}}, \sqrt{\frac{1-\beta}{4 CL}}}.
    \end{align*}

    \subsection{Proof of Lemma \ref{lem:lya}}
    \label{app:lya}

    We define the Lyapunov function $\cL_k$ as 
    \begin{equation}
        \label{eq:Lyapunov_can}
        \begin{aligned}
            \cL_k&:=\E\brk{f(\bar{d}_k) - f^*} + \alpha \cC_1 \cR_k^x + \alpha^3\cC_2 \cR_k^y + \alpha^3\cC_3\E\brk{\norm{\bar{z}_{k-1}}^2} \\
            &\quad + \alpha^3\cC_4\E\brk{\norm{\z_{k-1} - \nabla F(\Barx{k-1})}^2}, 
        \end{aligned}
    \end{equation}
    where $\cC_1$-$\cC_4$ are positive constants to be determined later. 
Combining the results of Lemmas \ref{lem:descent_dk}-\ref{lem:cons_R} and \eqref{eq:Lyapunov_can} leads to 
    \begin{align}
        \cL_{k+1} &\leq \brk{1 + \frac{2\alpha^2 CL}{n}+  12\alpha^3 \cC_2 nC(1-\beta)^2c_0 + \frac{2\alpha^3 C(1-\beta)^2 \cC_3}{n}\right.\nonumber\\
        &\left.\quad + 2n\alpha^3\cC_4 (1-\beta)^2C} \E\brk{f(\bar{d}_k) - f^*}\nonumber\\
        &\quad + \brk{\frac{\alpha L^2}{n}+ \alpha\cC_1\trw + \frac{15\alpha^3\cC_2 (1-\beta)^2 c_0L(L+C)}{1-\trw} + \frac{2\alpha^3 \cC_3(1-\beta)L(L+C)}{n}\right.\nonumber\\
        &\left.\quad + 2\alpha^3\cC_4(1-\beta)(L+C)L}\cR_k^x+ \brk{\frac{\alpha^3c_0\trw\cC_1}{1-\trw} + \frac{(1+\trw)\alpha^3\cC_2}{2}} \cR_k^y\nonumber\\
        &\quad + \brk{\frac{\alpha^3 \beta^2 L(L+2C)}{(1-\beta)^2} + 9\alpha^4\cC_2 n(C+L)(1-\beta)c_0 + \frac{(1+\beta)\alpha^3\cC_3}{2}\right.\nonumber\\
        &\left.\quad + 2\alpha^4\cC_4n(C+L)}\E\brk{\norm{\bar{z}_{k-1}}^2}\nonumber\\
        &\quad + \brk{\frac{3(1-\beta)^2\alpha^3\cC_2 c_0\beta}{1-\trw} +\alpha^3\cC_4\beta}\E\brk{\norm{\z_{k-1} - \nabla F(\Barx{k-1})}^2}\nonumber\\
        &\quad -\frac{\alpha}{2}\brk{1 - \frac{4\alpha^2 CL}{n(1-\beta)} - 40\alpha^3 n c_0(1-\beta)(C+L)\cC_2- 6(1-\beta)\alpha^2\cC_3 \right.\nonumber\\
        &\left.\quad - 4n(1-\beta)\alpha^3\cC_4C}\E\brk{\norm{\nabla f(\bar{x}_k)}^2}\nonumber\\
        &\quad + \brk{\frac{L}{n} + 6\alpha\cC_2 nc_0(1-\beta)^2 + \frac{\alpha \cC_3(1-\beta)^2}{n} + n\alpha\cC_4(1-\beta)^2}\alpha^2(2C\sigfn + \sigma^2).\label{eq:Lya_s1}
    \end{align}

    Therefore, it suffices to determine $\cC_1$-$\cC_4$ and the parameters $\alpha$ and $\beta$ such that 
    \begin{subequations}
        \begin{align}
            &\frac{\alpha L^2}{n}+ \frac{15\alpha^3\cC_2 (1-\beta)^2 c_0L(L+C)}{1-\trw} + \frac{2\alpha^3 \cC_3(1-\beta)L(L+C)}{n}\nonumber\\
            &\quad + 2\alpha^3\cC_4(1-\beta)(L+C)L\leq (1-\trw)\alpha \cC_1,\label{eq:c1} \\
            &\frac{\alpha^3c_0\trw\cC_1}{1-\trw} + \frac{(1+\trw)\alpha^3\cC_2}{2} \leq \alpha^3\cC_2, \label{eq:c2}\\
            &\frac{\alpha^3 \beta^2 L(L+2C)}{(1-\beta)^2} + 9\alpha^4\cC_2 n(C+L)(1-\beta)c_0 + \frac{(1+\beta)\alpha^3\cC_3}{2}\nonumber\\
            &\quad + 2\alpha^4\cC_4n(C+L) \leq \alpha^3\cC_3, \label{eq:c3}\\
            &\frac{3(1-\beta)^2\alpha^3\cC_2 c_0\beta}{1-\trw} +\alpha^3\cC_4\beta \leq \alpha^3\cC_4. \label{eq:c4}
        \end{align}
    \end{subequations}
    
    From \eqref{eq:c4}, it is sufficient to choose $\cC_4= 4c_0\cC_2$ given that $\beta\geq \trw$. According to \eqref{eq:c2}, we can choose $\cC_2 = 4c_0\trw\cC_1 / (1-\trw)^2$.
    Substituting the above $\cC_2$ and $\cC_4$ into \eqref{eq:c1} and \eqref{eq:c3} leads to 
    \begin{subequations}
        \begin{align}
            &\frac{L^2}{n}+ \frac{2\alpha^2\cC_3(1-\beta)L(L+C)}{n}\nonumber\\
            &\leq \prt{1-\trw -\frac{60\alpha^2(1-\beta)^2c_0^2\trw L(L+C)}{(1-\trw)^3} - \frac{32\alpha^2c_0^2(1-\beta)(L+C)L\trw}{(1-\trw)^2}}\cC_1, \label{eq:c1_s1}\\
            &\frac{\beta^2 L(L+2C)}{(1-\beta)^2} + \frac{36\alpha n(1-\beta)c_0^2C\cC_1\trw}{(1-\trw)^2} + \frac{8\alpha c_0^2 n(C+L)\cC_1\trw}{(1-\trw)^2} \leq \frac{(1-\beta)\cC_3}{2}.\label{eq:c3_s1}
        \end{align}
    \end{subequations} 
    Note that $(1-\beta)\leq (1-\trw)$. Letting the stepsize $\alpha$ satisfy 
    \begin{align}
        \label{eq:alpha_c1c3}
        \alpha\leq \min\crk{\sqrt{\frac{(1-\trw)^2}{240c_0^2L(L+C)}}, \sqrt{\frac{(1-\trw)^2}{128c_0^2L(L+C)}}}
    \end{align}
    leads to
    \begin{equation}
        \label{eq:c1_c3}
        \cC_1\geq \frac{2L^2}{n(1-\trw)} + \frac{4\alpha^2\cC_3(1-\beta)L(L+C)}{n(1-\trw)}.
    \end{equation}
According to \eqref{eq:c3_s1}, \eqref{eq:alpha_c1c3}, and \eqref{eq:c1_c3}, we can choose $\cC_1$ and $\cC_3$ as follows:
    \begin{equation}
        \label{eq:c1c3}
        \cC_1 = \frac{3L^2}{n(1-\trw)},\; \cC_3 = \frac{4L(L+2C)}{(1-\beta)^3}.
    \end{equation}
    To make such a choice feasible, it is sufficient to let the stepsize $\alpha$ and the parameter $\beta$ satisfy 
    \begin{align*}
        \alpha\leq\min\crk{\frac{1-\trw}{48 L}, \frac{1}{216(C+L)}},\; \beta\geq \trw,
    \end{align*}
    so that the following inequality, which is derived from \eqref{eq:c3_s1}, holds
    \begin{align*}
        \frac{108\alpha (1-\beta)c_0 CL^2}{(1-\trw)^3} + \frac{24\alpha c_0(C+L)L^2}{(1-\trw)^3}\leq \frac{L(L+2C)}{(1-\beta)^2}.
    \end{align*}

    Therefore, we have determined the constants $\cC_1$-$\cC_4$:
    \begin{align}
        \label{eq:cCs}
        \cC_1= \frac{3L^2}{n(1-\trw)},\; \cC_2 = \frac{12c_0\trw L^2}{n(1-\trw)^3}, \; \cC_3 = \frac{4L(L+2C)}{(1-\beta)^3}, \; \cC_4= \frac{48 c_0^2\trw L^2}{n(1-\trw)^3}.
    \end{align}
Hence, the inequality \eqref{eq:Lya_s1} becomes 
    \begin{equation}
        \label{eq:Lyapunov_s1}
        \begin{aligned}
            &\cL_{k+1}\leq \brk{1 + \frac{2\alpha^2 CL}{n}+  \frac{240\alpha^3 c_0^2\trw CL^2(1-\beta)^2}{(1-\trw)^3} + \frac{8\alpha^3 CL(L+2C)}{n(1-\beta)}}\cL_k\\
            &-\frac{\alpha}{2}\brk{1 - \frac{4\alpha^2 CL}{n(1-\beta)}- \frac{672\alpha^3c_0^2\trw(1-\beta)L^2(C+L)}{(1-\trw)^3} - \frac{24\alpha^2L(L+2C)}{(1-\beta)^2}}\E\brki{\normi{\nabla f(\bar{x}_k)}^2}\\
            &+ \brk{\frac{L}{n} + \frac{72\alpha c_0^2\trw L^2(1-\beta)^2}{(1-\trw)^3} + \frac{4\alpha L(L+2C)}{n(1-\beta)}}\alpha^2(2C\sigfn + \sigma^2).
        \end{aligned}
    \end{equation}
Letting the stepsize satisfy 
    \begin{align*}
        \alpha\leq\min\crk{\frac{1-\beta}{4(L+2C)}, \prt{\frac{(1-\trw)^3}{4032 c_0^2(1-\beta)L^2(C+L)}}^{1/3}, \sqrt{\frac{(1-\beta)^2}{144L(L+2C)}}, \sqrt{\frac{1-\beta}{24 CL}}}
    \end{align*}
    yields the desired result for the recursion of $\cL_k$. 

    We next consider the term $\cL_0$, for which we have 
    \begin{equation}
        \label{eq:cL0}
        \begin{aligned}
            \cL_0 &= f(\bar{x}_0) - f^* + \frac{3\alpha L^2}{n(1-\trw)}\cR_0^x + \frac{12\alpha^3c_0\trw L^2}{n(1-\trw)^3}\cR_0^y\\
            &\quad + \frac{48\alpha^3c_0^2\trw L^2}{n(1-\trw)^3}\norm{\z_{-1} - \nabla F(\Barx{-1})}^2\\
            &\leq \brk{1 + \frac{6\alpha^3c_0^2L^2C(1-\beta)^2}{(1-\trw)^2} + \frac{24\alpha^3c_0^2L^2C(1-\beta)^2}{(1-\trw)^3}}\Delta_0\\
            &\quad + \brk{c_0 + \frac{2\alpha^2c_0(1-\beta)^2L(L+C)}{1-\trw}+ \frac{8\alpha^2c_0^2L(L+C)(1-\beta)^2}{(1-\trw)^2} }\frac{3\alpha L^2\norm{\Pi\x_0}^2}{n(1-\trw)}\\
            &\quad  + \frac{27\alpha^3c_0^2 L^2(1-\beta)^2(2C\sigfn + \sigma^2)}{(1-\trw)^3}+ \frac{72\alpha^3c_0^2 L^2\sumn\norm{\nabla f_i(\bar{x}_0)}^2}{n(1-\trw)^3}\\
            &\leq 2\Delta_0 + \frac{9\alpha L^2\norm{\Pi\x_0}^2}{n(1-\trw)} + \frac{27\alpha^3 c_0^2L^2(1-\beta)^2(2C\sigfn + \sigma^2)}{(1-\trw)^2}\\
            &\quad + \frac{72\alpha^3c_0^2 L^2\sumn\norm{\nabla f_i(\bar{x}_0)}^2}{n(1-\trw)^3},
        \end{aligned}
    \end{equation}
    { by noting that}
    \begin{equation}
        \label{eq:g0}
        \begin{aligned}
            &\E\brk{\norm{\Pi\y_0}^2} \leq (1-\beta)^2\E\brk{\norm{\g_0}^2}\leq (1-\beta)^2\crk{2nC\brk{f(\bar{x}_0) - f^*} + 2nC\sigfn + n\sigma^2\right.\\
            &\left.\quad +2L(L+C)\norm{\Pi\x_0}^2 + 2\sumn\norm{\nabla f_i(\bar{x}_0)}^2},\\
            &\norm{\z_{-1} - \nabla F(\Barx{-1})}^2 = \norm{\nabla F(\Barx{0})}^2= \sumn\norm{\nabla f_i(\bar{x}_0)}^2,
        \end{aligned}
    \end{equation}
    { and the definition of $\cR_0^x$ and $\cR_0^y$ in \eqref{eq:cR0}.}

    \section{Proofs for the PL Condition Case}

    \subsection{Proof of Lemma \ref{lem:contract_dk}}
    \label{app:contract_dk}
 
    Similar to the derivation of \eqref{eq:de}, we have
        \begin{align}
            &\condE{f(\bar{d}_{k + 1}) - f^*}{\cF_k} \leq f(\bar{d}_k) - f^* - \alpha\inpro{\nabla f(\bar{d}_k), \frac{1}{n}\sumn\nabla f_i(x_{i,k})}\nonumber\\
            &\quad + \frac{L}{2}\condE{\norm{\alpha \bar{g}_k}^2}{\cF_k}\nonumber\\
            &\leq f(\bar{d}_k) - f^* - \frac{\alpha}{2}\norm{\nabla f(\bar{d}_k)}^2 - \frac{\alpha(1-\alpha L)}{2}\norm{\frac{1}{n}\sumn\nabla f_i(x_{i,k})}^2\nonumber\\
            &\quad + \frac{\alpha}{2}\norm{\nabla f(\bar{d}_k) - \frac{1}{n}\sumn\nabla f_i(x_{i,k})}^2+ \frac{\alpha^2 L}{2} \condE{\norm{\bar{g}_k - \frac{1}{n}\sumn\nabla f_i(x_{i,k})}^2}{\cF_k}\nonumber\\
            &\leq \prt{1 - \alpha\mu } \brk{f(\bar{d}_k) - f^*}  + \frac{\alpha L^2}{2n}\sumn\norm{\bar{d}_k - x_{i,k}}^2 + \frac{\alpha^2 L(2C\sigfn + \sigma^2)}{n}\nonumber\\
            &\quad + \frac{\alpha^2 LC}{n}\crk{\prt{1 + \frac{\alpha\beta L}{(1-\beta)}}\brk{f(\bar{d}_k) - f^*}  + \frac{\alpha\beta (1 + \alpha\beta L/(1-\beta))}{2(1-\beta)}\norm{\bar{z}_{k-1}}^2}\nonumber\\
            &\quad + \frac{\alpha^2 L^2C}{2n^2}\norm{\Pi\x_k}^2, \label{eq:de_pl}
        \end{align}
         where we invoke $L$-smoothness of $f_i$, Assumptions \ref{as:abc} and \ref{as:PL}, and \eqref{eq:fx_fd} for the last inequality.
         In the light of the relation $\bar{d}_k - \bar{x}_k = -\alpha\beta\bar{z}_{k-1}/(1-\beta)$, we have 
         \begin{equation}
             \label{eq:dxik}
             \begin{aligned}
                 \frac{1}{n}\sumn\norm{\bar{d}_k - x_{i,k}}^2&= \frac{\alpha^2\beta^2}{(1-\beta)^2}\norm{\bar{z}_{k-1}}^2 + \frac{1}{n}\norm{\Pi\x_k}^2.
             \end{aligned}
         \end{equation}
     Combining \eqref{eq:de_pl}-\eqref{eq:dxik}, letting $\alpha\leq (1-\beta)/L$ and taking the full expectation lead to 
         \begin{align*}
             \E\brk{f(\bar{d}_{k + 1}) - f^*} &\leq \prt{1 - \alpha\mu + \frac{2\alpha^2 L C}{n}}\E\brk{f(\bar{d}_k) - f^*} + \frac{\alpha L^2}{n}\E\brk{\norm{\Pi\x_k}^2}\\
             &\quad + \frac{\alpha^2 L(2C\sigfn + \sigma^2)}{n}+ \frac{\alpha^3 \beta L\brk{C(1-\beta) + L\beta}}{(1-\beta)^2}\E\brk{\norm{\bar{z}_{k-1}}^2}.
         \end{align*}
 Letting $\alpha \leq \mu/(4LC)$ yields the desired result.
 
     \subsection{Proof of Lemma \ref{lem:cL_pl}}
     \label{app:cL_pl}
 
     Similar to the derivation of \eqref{eq:Lya_s1}, we determine the new recursion of $\cL_k$ in light of Assumption \ref{as:PL} as follows:
    \begin{align}
        &\cL_{k+1} \leq \brk{1 - \frac{\alpha\mu}{2} +  \frac{240\alpha^3 c_0^2\trw CL^2}{1-\trw} + \frac{8\alpha^3 CL(L+2C)}{n(1-\beta)}} \E\brk{f(\bar{d}_k) - f^*}\nonumber\\
        &\quad + \brk{\frac{\alpha L^2}{n}+ \frac{3\alpha L^2\trw}{n(1-\trw)} + \frac{180\alpha^3 c_0^2\trw L^3(L+C)}{n(1-\trw)} + \frac{8\alpha^3 L^2(L+2C)(L+C)}{n(1-\beta)^2}\right.\nonumber\\
        &\left.\quad + \frac{96\alpha^3c_0^2\trw(L+C)L^3}{n(1-\trw)^2}}\cR_k^x + \brk{\frac{3\alpha^3c_0 L^2\trw}{n(1-\trw)^2} + \frac{12(1+\trw)\alpha^3c_0\trw L^2}{2n(1-\trw)^3}} \cR_k^y\nonumber\\
        &\quad + \brk{\frac{\alpha^3 \beta L(L+2C)}{(1-\beta)^2} + \frac{108\alpha^4 L^2 (C+L)c_0^2\trw}{(1-\trw)^2} + \frac{4(1+\beta)\alpha^3 L(L+2C)}{2(1-\beta)^3}\right.\nonumber\\
        &\left.\quad + \frac{96\alpha^4c_0^2\trw L^2(C+L)}{(1-\trw)^3}}\E\brk{\norm{\bar{z}_{k-1}}^2}\nonumber\\
        &\quad + \brk{\frac{36\alpha^3c_0^2 \trw \beta L^2(1-\beta)^2}{n(1-\trw)^4} + \frac{48\alpha^3c_0^2\trw L^2\beta}{n(1-\trw)^3}}\E\brk{\norm{\z_{k-1} - \nabla F(\Barx{k-1})}^2}\nonumber\\
        &\quad + \frac{48\alpha^3 L(L+C)}{(1-\beta)^2}\crk{4L\E\brk{f(\bar{d}_k) - f^*} + \frac{2\alpha^2\beta^2L^2}{(1-\beta)^2}\E\brk{\norm{\bar{z}_{k-1}}^2} }\nonumber\\
        &\quad + \brk{\frac{L}{n} + \frac{72\alpha c_0^2\trw L^2(1-\beta)^2}{(1-\trw)^3} + \frac{4\alpha L(L+2C)}{n(1-\beta)}}\alpha^2(2C\sigfn + \sigma^2), \label{eq:Lya2_s1}
    \end{align}
         where we consider $\alpha\leq 1/(28c_0^2 L)$, $\beta\geq \trw$, and the estimate in \eqref{eq:nx_fx}. The estimate \eqref{eq:nx_fx} holds due to \eqref{eq:db_xb}  
         \begin{equation}
             \label{eq:nx_fx}
             \begin{aligned}
                 \norm{\nabla f(\bar{x}_k)}^2&\leq 2\norm{\nabla f(\bar{d}_k)}^2 + \frac{2\alpha^2\beta^2L^2}{(1-\beta)^2}\norm{\bar{z}_{k-1}}^2\\
                 &\leq 4L\brk{f(\bar{d}_k) - f^*} + \frac{2\alpha^2\beta^2L^2}{(1-\beta)^2}\norm{\bar{z}_{k-1}}^2.
             \end{aligned}
         \end{equation}
 
         To derive the result in \eqref{eq:cL_pl}, it is then sufficient to determine $\alpha$ and $\beta$ such that 
 
         \begin{subequations}
             \begin{align}
                 &\frac{240\alpha^3 c_0^2\trw CL^2}{1-\trw} + \frac{8\alpha^3 CL(L+2C)}{n(1-\beta)} + \frac{192\alpha^3 L^2(L+C)}{(1-\beta)^2}\leq \frac{\alpha\mu}{4}, \label{eq:a1}\\
                 &12\alpha^2c_0^2\trw^3 LC + \frac{8\alpha^2(L+2C)(L+C)(1-\trw)}{3(1-\beta)^2} + \frac{32\alpha^2 c_0^2\trw(L+C)L}{1-\trw} \nonumber\\
                 &\leq \frac{2(1-\trw)}{3} - \frac{\alpha\mu}{4}, \label{eq:a2}\\
                 &0\leq \frac{1-\trw}{4} - \frac{\alpha\mu}{4}, \label{eq:a3}\\
                 & \frac{27\alpha L(C+L) c_0^2\trw (1-\beta)}{L+2C} + \frac{24\alpha c_0^2\trw L(C+L)}{L+2C}+ \frac{24\alpha^2(L+C)L^2\beta^2}{(L+2C)(1-\beta)} \nonumber\\
                 & \leq \frac{1-\beta}{4} - \frac{\alpha\mu}{4}, \label{eq:a4}\\
                 &\frac{3(1-\beta)^2}{4(1-\trw)} \leq 1 - \beta - \frac{\alpha\mu}{4}\label{eq:a5}.
             \end{align}
         \end{subequations}
 To ensure that~\eqref{eq:a1} holds, it suffices to have 
     \begin{align*}
         \frac{240\alpha^3c_0^2L(L+C)^2}{(1-\beta)^2}\leq \frac{\alpha \mu}{4},\; \beta\geq 1- \frac{11}{12}(1-\trw).
     \end{align*}
     Then, we let $\alpha\leq \sqrt{\mu(1-\beta)^2/[960c_0^2 L(L+C)^2]}$.
 
 
     For \eqref{eq:a2} to hold, it suffices to let the stepsize $\alpha$ satisfy
     \begin{align*}
    \alpha\leq\min\crk{\frac{4(1-\trw)}{3\mu}, \sqrt{\frac{1-\trw}{108 c_0^2LC}}, \sqrt{\frac{(1-\trw)^2}{288c_0L(L+C)}}, \sqrt{\frac{(1-\beta)^2}{48(L+C)^2}}}.
     \end{align*}
     For \eqref{eq:a3}, it is sufficient to let $\alpha\leq (1-\trw)/\mu$. For \eqref{eq:a4}, it suffices to have 
     \begin{align*}
         \alpha\leq\min\crk{\frac{1-\beta}{3\mu}, \frac{1}{486c_0^2 L}, \frac{1-\beta}{432c_0L}, \sqrt{\frac{(1-\beta)^2}{432 L^2}}}.
     \end{align*}
Finally, we choose $\beta \geq (1 + 11\trw)/12$ such that \eqref{eq:a5} holds and obtain the desired result.
 
     \subsection{Proof of Theorem \ref{thm:pl}}
     \label{app:pl}
 
         We first derive an upper bound for $\E\brki{f(\bar{d}_k) - f^*}$. Unrolling the recursion \eqref{eq:cL_pl} in Lemma \ref{lem:cL_pl} and rearranging the terms yield 
         \begin{equation}
             \label{eq:fdbar_bound}
             \begin{aligned}
                 &\E\brk{f(\bar{d}_k) - f^*}\leq \cL_{k} \\
                 &\leq \prt{1 - \frac{\alpha\mu}{4}}^k\cL_0 +\brk{\frac{L}{n} + \frac{72\alpha c_0^2\trw L^2(1-\beta)^2}{(1-\trw)^3} + \frac{4\alpha L(L+2C)}{n(1-\beta)}}\frac{4\alpha(2C\sigfn + \sigma^2)}{\mu}\\
                 &\leq \prt{1 - \frac{\alpha\mu}{4}}^k\cL_0 + \frac{4\alpha L(2C\sigfn + \sigma^2)}{n\mu} + \frac{16\alpha^2 L^2\cS_1 (2C\sigfn + \sigma^2)}{(1-\trw)\mu},
             \end{aligned}
         \end{equation}
         where 
         \begin{align*}
             \cS_1:= \frac{18c_0^2(1-\beta)^2}{(1-\trw)^2} + \frac{(L+2C)(1-\trw)}{nL (1-\beta)}.
         \end{align*}
         We next construct a new Lyapunov function $\cH_k$ in \eqref{eq:hk_can} such that the consensus error $\E\brki{\normi{\Pi\x_k}^2}$ can be bounded.
         \begin{equation}
             \label{eq:hk_can}
             \cH_k:= \cR_k^x + \alpha^2 \cC_5 \cR_k^y + \alpha^3 \cC_6 \E\brk{\norm{\bar{z}_{k-1}}^2} + \alpha^2 \cC_7\E\brk{\norm{\z_{k-1} -\nabla F(\Barx{k-1})}^2}.
         \end{equation}
Combining the results of Lemmas \ref{lem:zk}-\ref{lem:cons_R}, invoking \eqref{eq:nx_fx}, and rearranging the terms yield

            \begin{align}
                &\cH_{k + 1} \leq \brk{\trw + \frac{15\alpha^2\cC_5(1-\beta)^2c_0L(C+L)}{1-\trw} + \frac{2\alpha^3(1-\beta)L(L+C)\cC_6}{n}\right.\nonumber\\
                &\left.\quad + 2\alpha^2\cC_7(1-\beta)L(L+C)}\cR_k^x + \brk{\frac{\alpha^2c_0\trw}{1-\trw} + \frac{\alpha^2\cC_5(1+\trw)}{2}}\cR_k^y\nonumber\\
                &\quad + \brk{9\alpha^3\cC_5 nc_0(1-\beta)(C+L) + \frac{(1+\beta)\alpha^3\cC_6}{2} + 2\alpha^3\cC_7n(C+L)}\E\brk{\norm{\bar{z}_{k-1}}^2}\nonumber\\
                &\quad + \brk{\frac{3\alpha^2\cC_5(1-\beta)^2c_0\beta}{1-\trw} + \beta\alpha^2\cC_7}\E\brk{\norm{\z_{k-1} - \nabla F(\Barx{k-1})}^2}\nonumber\\
                &\quad + \brk{12\alpha^2\cC_5nc_0(1-\beta)^2C + \frac{2\alpha^3\cC_6 C(1-\beta)^2}{n} + 2\alpha^2\cC_7n(1-\beta)^2C}\cL_k\nonumber\\
                &\quad + \brk{20\alpha^3\cC_5nc_0(1-\beta)(C+L) + 3\alpha^3\cC_6(1-\beta) + 2\alpha^3\cC_7n(1-\beta)C}\nonumber\\
                &\quad \cdot \crk{4L\cL_k + \frac{2\alpha^2\beta^2 L^2}{(1-\beta)^2}\E\brk{\norm{\bar{z}_{k-1}}^2}}\nonumber\\
                &\quad + \brk{6\cC_5nc_0(1-\beta)^2 + \frac{\alpha \cC_6(1-\beta)^2}{n} + \cC_7 n(1-\beta)^2}\alpha^2(2C\sigfn + \sigma^2). \label{eq:hk_s1}
            \end{align}

         Letting $\alpha\leq (1-\beta)/(5L)$, it is sufficient to determine $\cC_5$-$\cC_7$, and $\beta$ such that 
         \begin{subequations}
             \begin{align}
                 & \trw + \frac{15\alpha^2\cC_5(1-\beta)^2c_0L(C+L)}{1-\trw} + \frac{2\alpha^3(1-\beta)L(L+C)\cC_6}{n}\nonumber\\
                 &\quad + 2\alpha^2\cC_7(1-\beta)L(L+C) \leq \frac{5+\beta}{6},\label{eq:abeta}\\
                 & \frac{\alpha^2c_0\trw}{1-\trw} + \frac{\alpha^2\cC_5(1+\trw)}{2} \leq \frac{5 + \beta}{6}\alpha^2\cC_5,\label{eq:c5_can}\\
                 & 9\alpha^3\cC_5 nc_0(1-\beta)(C+L) + \frac{(1+\beta)\alpha^3\cC_6}{2} + 2\alpha^3\cC_7n(C+L)\leq \frac{5 + \beta}{6}\alpha^3\cC_6, \label{eq:c6_can}\\
                 & \frac{3\alpha^2\cC_5(1-\beta)^2c_0\beta}{1-\trw} + \alpha^2\cC_7\beta \leq \frac{5 + \beta}{6}\alpha^2\cC_7\label{eq:c7_can}.
             \end{align}
         \end{subequations}
 
         For \eqref{eq:c7_can}, we can choose $\cC_7 = 18c_0\cC_5$. For \eqref{eq:c5_can} it suffices  that  
         \begin{align*}
             \frac{c_0\trw}{1-\trw}\leq \prt{\frac{5+\beta}{6} - \frac{1+\trw}{2}}\cC_5.
         \end{align*}
 
         We can then choose 
         \begin{align*}
             \cC_5:= \frac{3c_0\trw}{(1-\trw)^2},\; \cC_7:= \frac{54c_0^2\trw}{(1-\trw)^2},\; \beta\geq \trw.
         \end{align*}
 
         Substituting $\cC_5$ and $\cC_7$ into \eqref{eq:c6_can} yields
 
         \begin{equation}
             \label{eq:c6_can1}
             \frac{27c_0^2\trw n C(1-\beta)}{(1-\trw)^2} + \frac{108c_0^2 \trw n (C+L)}{(1-\trw)^2}\leq \frac{(1-\beta)\cC_6}{3}.
         \end{equation}
 
         Therefore, we can choose 
         \begin{equation}
             \label{eq:c6}
             \cC_6:= \frac{405nc_0^2(C+L)}{(1-\trw)^3},\; \beta\geq \trw.
         \end{equation}
 
         Substituting $\cC_5-\cC_7$ into \eqref{eq:abeta} and noting $\beta\geq \trw$ lead to 
         \begin{equation}
             \label{eq:abeta_s1}
             \begin{aligned}
                &\frac{45\alpha^2(1-\beta)^2c_0^2L(C+L)\trw}{(1-\trw)^3} + \frac{910c_0^2\alpha^3(1-\beta)L(L+C)^2}{(1-\trw)^3}\\
                &\quad + \frac{108c_0^2\trw\alpha^2(1-\beta)L(L+C)}{(1-\trw)^2} \leq \frac{5(1-\trw)}{6}.
             \end{aligned}
         \end{equation}
 
         It suffices to let $\alpha\leq (1-\trw)/[160c_0(L+C)]$ for \eqref{eq:abeta_s1} to hold. Then, the inequality \eqref{eq:hk_s1} becomes 
         \begin{equation}
             \begin{aligned}
                 \label{eq:hk_re}
                 &\cH_{k + 1} \leq 
                 \frac{5 + \beta}{6}\cH_k + \frac{72nc_0^2\alpha^2(1-\beta)^2(2C\sigfn + \sigma^2)}{(1-\trw)^2}\brk{1+ \frac{45\alpha(C+L)}{8(1-\trw)}}\\
                 & + \brk{\frac{144\alpha^2nc_0^2(1-\beta)^2(C+L)(1 + 5\alpha L/(1-\beta))}{(1-\trw)^2} + \frac{6370\alpha^3 n c_0^2(1-\beta)C(C+L)}{(1-\trw)^3}}\cL_k\\
                 &\leq \frac{5 + \beta}{6}\cH_k + \frac{168\alpha^2(1-\beta)^2nc_0^2(2C\sigfn + \sigma^2)}{(1-\trw)^2} + \frac{288\alpha^2c_0^2(1-\beta)^2n(C + L)\cL_k}{(1-\trw)^2},
             \end{aligned}
         \end{equation}
         where the last inequality follows from $\alpha\leq (1-\beta)/[234(C+L)]$.
 
         Unrolling \eqref{eq:hk_re} and substituting \eqref{eq:fdbar_bound} yield
 
         \begin{equation}
             \label{eq:hk}
             \begin{aligned}
                 \cH_k &\leq \prt{\frac{5+\beta}{6}}^k\cH_0 + \frac{1008\alpha^2(1-\beta)nc_0^2(2C\sigfn + \sigma^2)}{(1-\trw)^2}\\
                 &\quad + \frac{288\alpha^2c_0^2(1-\beta)^2n(C + L)\cL_0}{(1-\trw)^2}\prt{1 - \frac{\alpha\mu}{4}}^k\sum_{j=0}^k\prt{\frac{1 - \frac{1-\beta}{6}}{1-\frac{\alpha\mu}{4}}}^j\\
                 &\quad + \frac{1728\alpha^2(1-\beta)c_0^2n(C + L)}{(1-\trw)^2}\crk{\frac{4\alpha L(2C\sigfn + \sigma^2)}{n\mu} + \frac{16\alpha^2 L^2\cS_1 (2C\sigfn + \sigma^2)}{(1-\trw)\mu}}\\
                 &\leq \prt{\frac{5 + \beta}{6}}^{k}\cH_0 + \frac{25 n\cL_0}{L+C}\prt{1 - \frac{\alpha\mu}{4}}^k + \frac{1008\alpha^2(1 + \cS_2)(1-\beta)nc_0^2(2C\sigfn + \sigma^2)}{(1-\trw)^2},
             \end{aligned}
         \end{equation}
         where the following relations are noted: 
         \begin{equation}
             \label{eq:rsum}
             \begin{aligned}
                 &\sum_{j=0}^{k}\prt{\frac{1-\frac{1-\beta}{6}}{1-\frac{\alpha\mu}{4}}}^j\leq \sum_{j=0}^{k}\prt{\frac{10+2\beta}{11+\beta}}^j\leq \frac{12}{1-\beta},\; \cS_2:= \frac{L}{n\mu} + \frac{\cS_1 L}{\mu},\\
                 &\alpha\leq\min\crk{\frac{1-\beta}{4L}, \frac{1}{12(C+L)}}.
             \end{aligned}
         \end{equation}

         Substituting \eqref{eq:fdbar_bound} and \eqref{eq:hk} into the result of Lemma \ref{lem:zk} and invoking \eqref{eq:nx_fx} yield that, for $\alpha\leq (1-\beta)/(2L\sqrt{3})$, there holds
         \begin{equation}
             \label{eq:zkbar_bound}
             \begin{aligned}
                 &\E\brk{\norm{\bar{z}_k}^2}
                 \leq \frac{3+\beta}{4}\E\brk{\norm{\bar{z}_{k-1}}^2} + \frac{(1-\beta)^2(2C\sigfn + \sigma^2)}{n}\\
                 &\quad +2(1-\beta)(C+6L)\crk{\prt{1 - \frac{\alpha\mu}{4}}^k\cL_0 + \frac{4\alpha L(2C\sigfn + \sigma^2)}{n\mu}\right.\\
                 &\left.\quad + \frac{16\alpha^2 L^2\cS_1 (2C\sigfn + \sigma^2)}{(1-\trw)\mu}}+ \frac{2(1-\beta)L(L+C)}{n}\crk{\prt{\frac{5 + \beta}{6}}^{k}\cH_0\right.\\ 
                 &\left.\quad + \frac{25 n\cL_0}{L}\prt{1 - \frac{\alpha\mu}{4}}^k + \frac{1008\alpha^2(1 + \cS_2)(1-\beta)nc_0^2(2C\sigfn + \sigma^2)}{(1-\trw)^2}}\\ 
                 &\leq \frac{24L(L+C)\cH_0}{n}\prt{\frac{5+\beta}{6}}^k + 24\cL_0(C+L)\prt{1-\frac{\alpha\mu}{4}}^k + \frac{4(1-\beta)(2C\sigfn + \sigma^2)}{n}\\
                 &\quad + \brk{\frac{6L+C}{n\mu} + \frac{4\alpha L\cS_1(C+6L)}{(1-\trw)\mu} + \frac{252\alpha(1+\cS_2)(1-\beta)(L+C)}{(1-\trw)^2}}32\alpha L(2C\sigfn + \sigma^2).
             \end{aligned}
         \end{equation}      
         From Lemma \ref{lem:fx_fd}, we have 
         \begin{align*}
             \prt{1 - \frac{2\alpha\beta L}{1-\beta}}\E\brk{f(\bar{x}_k) - f^*} &\leq \E\brk{f(\bar{d}_k) - f^*} + \frac{2\alpha\beta}{1-\beta}\E\brk{\norm{\bar{z}_{k-1}}^2}.
         \end{align*}
 Therefore, we have for $\alpha \leq (1-\beta)/(4\beta L)$ that 
         \begin{equation}
             \label{eq:pl_goal}
             \begin{aligned}
                 \frac{1}{n}\sumn\E\brk{f(x_{i,k}) - f^*} &\leq 2\E\brk{f(\bar{x}_k) - f^*} + \frac{L}{n}\E\brk{\norm{\Pi\x_k}^2}\\
                 &\leq 4\cL_k + \frac{4\alpha\beta}{1-\beta}\E\brk{\norm{\bar{z}_{k-1}}^2} + \frac{L}{n}\cH_k.
             \end{aligned}
         \end{equation}
 Substituting \eqref{eq:fdbar_bound}, \eqref{eq:hk}, and \eqref{eq:zkbar_bound} into \eqref{eq:pl_goal} and setting
         \begin{align*}
             \alpha\leq \min\crk{\frac{1-\beta}{32(C+L)}, \frac{1-\beta}{4(C+6L)}}
         \end{align*}
         yields the desired result.

\end{appendices}


\bibliography{references_all}

\end{document}